\documentclass[oneside,english]{amsart}
\usepackage[T1]{fontenc}
\usepackage[latin9]{inputenc}
\usepackage{amsbsy}
\usepackage{amsthm}
\usepackage{amssymb}
\usepackage{graphicx}
\usepackage{esint}

\makeatletter
\numberwithin{equation}{section}
\numberwithin{figure}{section}
\theoremstyle{plain}
\newtheorem{thm}{\protect\theoremname}
  \theoremstyle{definition}
  \newtheorem{defn}[thm]{\protect\definitionname}
  \theoremstyle{remark}
  \newtheorem{rem}[thm]{\protect\remarkname}
  \theoremstyle{definition}
  \newtheorem{example}[thm]{\protect\examplename}
  \theoremstyle{plain}
  \newtheorem{lem}[thm]{\protect\lemmaname}
  \theoremstyle{plain}
  \newtheorem{cor}[thm]{\protect\corollaryname}
  \theoremstyle{plain}
  \newtheorem*{lem*}{\protect\lemmaname}

\makeatother

\usepackage{babel}
  \providecommand{\corollaryname}{Corollary}
  \providecommand{\definitionname}{Definition}
  \providecommand{\examplename}{Example}
  \providecommand{\lemmaname}{Lemma}
  \providecommand{\remarkname}{Remark}
\providecommand{\theoremname}{Theorem}

\begin{document}

\title{Decay Rate of Iterated Integrals of Branched Rough Paths }

\author{Horatio Boedihardjo }
\begin{abstract}
Iterated integrals of paths arise frequently in the study of the Taylor's
expansion for controlled differential equations. We will prove a factorial
decay estimate, conjectured by M. Gubinelli, for the iterated integrals
of non-geometric rough paths. We will explain, with a counter example,
why the conventional approach of using the neoclassical inequality
fails. Our proof involves a concavity estimate for sums over rooted
trees and a non-trivial extension of T. Lyons' proof in 1994 for the
factorial decay of iterated Young's integrals. 
\end{abstract}

\address{Department of Mathematics and Statistics, Reading University, UK. }


\thanks{A substantial part of this work was carried out while the author
was at Oxford-Man Institute, Oxford University, UK. We are grateful
to the support of ERC grant Esig (agreement no. 291244 ) during that
period. We would also like to thank M. Gubinelli and D. Yang for the
useful discussions, as well as the anonymous referee for the detailed
comments. }

\maketitle

\section{Introduction }

The iterated integrals of a path arise naturally from the Taylor's
expansion of a controlled differential equation driven by the path
and play a fundamental role in the theory of rough paths \cite{Lyons98}.
Given a path $x$, we are interested in the behaviour of the iterated
integral 
\begin{equation}
X_{0,1}^{n}=\int_{0<s_{1}<\ldots<s_{n}<1}\mathrm{d}x_{s_{1}}\otimes\ldots\otimes\mathrm{d}x_{s_{n}}\label{eq:iterated integrals}
\end{equation}
as $n$ varies. The solution for a linear controlled differential
equation has a series expansion that is linear in these iterated integrals.
The convergence of the series expansion is often studied using the
decay of these iterated integrals. The simplest example is the case
when $x$ takes value in $\mathbb{R}^{d}$ and has an almost everywhere
derivative in $L^{\infty}$, in which case 
\[
\Vert X_{0,1}^{n}\Vert\leq\frac{\left|\dot{x}\right|_{L^{\infty}}^{n}}{n!}.
\]
The problem becomes much harder when $x$ does not have a derivative,
such as in the case when the iterated integral (\ref{eq:iterated integrals})
is defined in terms of Young's integration. It was proved by Lyons
\cite{Lyons94} that if $x$ is $\gamma$-Hölder, $\gamma>\frac{1}{2}$,
and $\left\Vert x\right\Vert _{\gamma}$ denotes the $\gamma$-Hölder
norm of $x$, then 
\begin{equation}
\Vert X_{0,1}^{n}\Vert\leq\left(1+\zeta\left(2\gamma\right)\right)^{n-1}\frac{\left\Vert x\right\Vert _{\gamma}^{n}}{n!^{\gamma}},\label{eq:lyons 94 estimate}
\end{equation}
where $\zeta$ is the classical Riemann Zeta function. For $0<\gamma\leq\frac{1}{2}$
and $N=\lfloor\gamma^{-1}\rfloor$, a $\gamma$ -Hölder geometric
rough path takes value in the unital tensor algebra 
\[
T^{(N)}\left(\mathbb{R}^{d*}\right)=1\oplus\mathbb{R}^{d*}\ldots\oplus\left(\mathbb{R}^{d*}\right)^{\otimes N}
\]
where $\mathbb{R}^{d*}$ denotes the dual of $\mathbb{R}^{d}$. Lyons
\cite{Lyons98} showed that a $\gamma$ -Hölder rough path $x$ can
be extended uniquely to a $\gamma$ -Hölder path $\mathbf{X}$ in
$T^{(n)}\left(\mathbb{R}^{d*}\right)$ for any $n\geq N$. He defined
the $n$-th order iterated integrals of $x$ up to time $1$ as the
$n$-th tensor component of this extended path at time $1$, which
we will denote for latter use as $X_{0,1}^{n}$, and showed that
\[
\left\Vert X_{0,1}^{n}\right\Vert \leq\gamma^{-n}\left(1+2^{\left(N+1\right)\gamma}\zeta\left(\left(N+1\right)\gamma\right)\right)^{n}\frac{\left\Vert x\right\Vert _{\gamma}^{n}}{\Gamma\left(n\gamma+1\right)},
\]
where $\Gamma$ is the Gamma function and $\left\Vert x\right\Vert _{\gamma}$
denotes the $\gamma$-Hölder norm of the rough path $x$. 

Recently Gubinelli \cite{RamificationofBranchedRoughPath10} proposed
a \textit{non-geometric }theory of rough path, known as the branched
rough paths. The phrase ``non-geometric'' here refers to that the
calculus with respect to branched rough paths does not have to satisfy
the chain rule 
\[
\mathrm{d}\left(XY\right)=Y\mathrm{d}X+X\mathrm{d}Y
\]
which Lyons' geometric rough paths must satisfy. For the Brownian
motion $B$, the rough path (almost surely defined) 
\[
\left(s,t\right)\rightarrow\left(1,\int_{s}^{t}\mathrm{d}B_{s_{1}},\int_{s}^{t}\int_{s}^{u_{1}}\mathrm{d}B_{s_{1}}\otimes\mathrm{d}B_{s_{2}}\right)
\]
is geometric if the integration is defined in the sense of Stratonovich
and non-geometric if the integration is defined in the sense of Itô.
Branched rough paths are indexed by the Connes-Kremier Hopf algebra
of labelled rooted trees, which we will denote by $\mathcal{H}_{\mathcal{L}}$
and will recall in Section 2. The multiplication of trees in $\mathcal{H}_{\mathcal{L}}$
corresponds to the multiplication of the coordinate components of
the path, while the operation of joining forests to a single root
corresponds to integrating against the path. The theory of branched
rough paths is a rough path analogue Butcher's tree-indexed series
expansions of solutions to differential equations and has also been
motivated by expansions in stochastic partial differential equations.
We now recall an equivalent definition of branched rough path due
to Hairer-Kelly \cite{HairerKelly12}.
\begin{defn}
(\cite{RamificationofBranchedRoughPath10}, \cite{HairerKelly12})\label{branched rough path definition}Let
$0<\gamma\leq1$. Let $\left(\mathcal{H}_{\mathcal{L}},\cdot,\triangle,S\right)$
be the Connes-Kremier Hopf algebra of rooted trees labelled by a finite
set $\mathcal{L}$. Let $\left(\mathcal{H}_{\mathcal{L}}^{*},\star,\delta,s\right)$
be the dual Hopf algebra of $\left(\mathcal{H}_{\mathcal{L}},\cdot,\triangle,S\right)$.
A $\gamma$-branched rough path is a map $X:\left[0,1\right]\times\left[0,1\right]\rightarrow\mathcal{H}_{\mathcal{L}}^{*}$
such that 

1. for all $s\leq t$ and all $h_{1},h_{2}\in\mathcal{H}_{\mathcal{L}}$,
\begin{equation}
\langle X_{s,t},h_{1}\rangle\langle X_{s,t},h_{2}\rangle=\langle X_{s,t},h_{1}\cdot h_{2}\rangle.\label{eq:tree homomorphism}
\end{equation}

2. for all $u\leq s\leq t$, 
\begin{equation}
X_{u,s}\star X_{s,t}=X_{s,t}.\label{eq:tree multiplicative}
\end{equation}

3.for all labelled rooted tree $\tau$, if $|\tau|$ denote the number
of vertices in $\tau$, then 
\begin{equation}
\left\Vert X\right\Vert _{\gamma,\tau}:=\sup_{s\neq t}\frac{|\langle X_{s,t},\tau\rangle|}{\left|t-s\right|^{\gamma|\tau|}}<\infty.\label{eq:Holder}
\end{equation}
\end{defn}
\begin{rem}
Hairer-Kelly \cite{HairerKelly12} pointed out that the product $\star$
in $\mathcal{H}_{\mathcal{L}}^{*}$ is induced by the coproduct $\triangle$
on $\mathcal{H}_{\mathcal{L}}$ in the following sense: If $h\in\mathcal{H}_{\mathcal{L}}$
is a rooted tree and $\triangle h=\sum\boldsymbol{h^{\left(1\right)}}\otimes h^{\left(2\right)}$,
then for $X,Y\in\mathcal{H}_{\mathcal{L}}^{*}$, 
\[
\langle X\star Y,h\rangle=\sum\langle X,\boldsymbol{h^{(1)}}\rangle\langle Y,h^{(2)}\rangle.
\]
Hairer-Kelly also realised that condition 1. in the Definition \ref{branched rough path definition}
of branched rough path is equivalent to $X$ taking value in the group
of characters of the Hopf algebra, known as the\textit{ Butcher group},
analogous to the nilpotent Lie group in the geometric case. \end{rem}
\begin{example}
\label{extending a branched rough path}Let $\frac{1}{2}<\gamma\leq1$,
and $x=\left(x^{1},\ldots,x^{d}\right):\left[0,1\right]\rightarrow\mathbb{R}^{d}$
be a $\gamma$-Hölder path in the sense that 
\[
\sup_{s\neq t}\frac{\left\Vert x_{t}-x_{s}\right\Vert }{\left|t-s\right|^{\gamma}}<\infty.
\]
The $\gamma>\frac{1}{2}$ assumption allows us to use Young's integration.
According to Gubinelli \cite{RamificationofBranchedRoughPath10},
we may lift $x$ to a $\gamma$-branched rough path $X$ in the following
way: 

Let $\bullet_{i}$ be a vertex labelled by $i\in\left\{ 1,\ldots,d\right\} $.
Let $\tau_{1},\ldots,\tau_{n}$ be rooted trees labelled by $\left\{ 1,\ldots,d\right\} $,
and $[\tau_{1},\ldots,\tau_{n}]_{\bullet_{i}}$ denote the labelled
tree obtained by connecting the roots of $\tau_{1},\ldots,\tau_{n}$
to the labelled vertex $\bullet_{i}$. Then $X$ is defined inductively
by $\langle X_{s,t},\bullet_{i}\rangle=x_{t}^{i}-x_{s}^{i}$ and 
\begin{equation}
\langle X_{s,t},[\tau_{1},\ldots,\tau_{n}]_{\bullet_{i}}\rangle=\int_{s}^{t}\Pi_{j=1}^{n}\langle X_{s,u},\tau_{j}\rangle\mathrm{d}x_{u}^{i}.\label{eq:product and joinning}
\end{equation}
We now explain why the integration in (\ref{eq:product and joinning})
can be defined in the sense of Young: if $u\rightarrow\langle X_{s,u},\tau_{j}\rangle$
is $\gamma$-Hölder, then the product $\Pi_{j=1}^{n}\langle X_{s,u},\tau_{j}\rangle$
is also $\gamma$-Hölder. By for example Theorem 1.16 in \cite{LCL07},
the integral in (\ref{eq:product and joinning}) is also $\gamma$-Hölder.
More generally, in \cite{HairerKelly12}, Hairer-Kelly gave an explicit
way of extending a $\gamma$-geometric rough path to a $\gamma$-branched
rough path.

For general $0<\gamma\leq1$, let $N=\lfloor\gamma^{-1}\rfloor$,
then by Theorem 7.3 in \cite{RamificationofBranchedRoughPath10},
given a family of real-valued functions $\left(\langle X_{\cdot,\cdot},\tau\rangle\right)_{\tau\in\mathcal{H}_{\mathcal{L}},\left|\tau\right|\leq N}$
on $\left[0,1\right]\times\left[0,1\right]$ satisfying the conditions
(1),(2) and (3) in Definition \ref{branched rough path definition}
of branched rough path, there is a unique way of extending $\left(\langle X_{\cdot,\cdot},\tau\rangle\right)_{\left|\tau\right|\leq N}$
to a $\gamma$-branched rough path $\left(\langle X_{\cdot,\cdot},\tau\rangle\right)_{\tau\in\mathcal{H}_{\mathcal{L}},\left|\tau\right|\geq0}$.
Gubinelli \cite{RamificationofBranchedRoughPath10} conjectured that
this extension, which can be interpreted as the iterated integrals
of the truncated branched rough path $\left(\langle X_{\cdot,\cdot},\tau\rangle\right)_{\tau\in\mathcal{H},\left|\tau\right|\leq N}$
has a tree factorial decay. Our main result, stated below, is a proof
of this conjecture. \end{example}
\begin{thm}
\label{thm:main result}Let $0<\gamma\leq1$ and $N=\lfloor\gamma^{-1}\rfloor$.
Let $X$ be a $\gamma-$branched rough path. For all rooted trees
$\tau$ and all $s\leq t$, 
\begin{equation}
|\langle X_{s,t},\tau\rangle|\leq\frac{\overline{c}_{N}^{\left|\tau\right|}(t-s)^{\gamma|\tau|}}{\tau!^{\gamma}}.\label{eq:main result bound}
\end{equation}
where 
\begin{eqnarray*}
\overline{c}_{N} & = & 6\exp\left(7\sum_{i=0}^{N+1}\left(N+1\right){}^{i+1}\right)\left|\mathcal{T}^{N}\right|^{2-2\gamma}2^{\left(N+1\right)\gamma}\zeta\left(\left(N+1\right)\gamma\right)N!^{\gamma}\max_{1\leq\left|\sigma\right|\leq N}\left\Vert X\right\Vert _{\gamma,\sigma}^{\left|\sigma\right|^{-1}},
\end{eqnarray*}
$\left\Vert X\right\Vert _{\gamma,\sigma}$ is the Hölder norm of
$X$ as defined in (\ref{eq:Holder}) and $\mathcal{T}^{N}$ is the
set of unlabelled rooted trees with at most $N$ vertices.\end{thm}
\begin{rem}
For $\gamma=1$, Gubinelli \cite{RamificationofBranchedRoughPath10}
showed that the decay rate in (\ref{eq:main result bound}) is attained
for the identity path $X$, defined for all rooted trees $\tau$ by
\[
\langle X_{s,t},\tau\rangle=\frac{\left(t-s\right)^{\left|\tau\right|}}{\tau!}.
\]

\end{rem}

\begin{rem}
Gubinelli \cite{Gubinelli Navier Stokes} used a similar type of factorial
decay estimate to prove the convergence of his series expansion for
the solution of the three-dimensional Navier-Stokes equation for sufficiently
small initial data. 
\end{rem}
Theorem \ref{thm:main result}, together with the Hairer-Kelly result
\cite{HairerKelly12} that geometric rough paths are branched rough
paths, gives another proof for the factorial decay for geometric rough
paths. In some cases, our main result gives a sharper estimate for
the shuffled sum of the iterated integrals than the one derived using
shuffle product and the factorial decay for geometric rough paths,
as the following example demonstrates. 
\begin{example}
Let $x$ and $y$ be real valued $\gamma-$Hölder paths on $\left[0,1\right]$,
$\frac{1}{2}<\gamma\leq1$. Then we may estimate absolute value of
the integral 
\[
\int_{0}^{1}x_{s}^{n}\mathrm{d}y_{s}
\]
by the ``geometric method''
\begin{eqnarray}
\left|\int_{0}^{1}x_{s}^{n}\mathrm{d}y_{s}\right| & = & n!\left|\int_{0}^{1}\int_{0}^{s}\ldots\int_{0}^{s_{2}}\mathrm{d}x_{s_{1}}\ldots\mathrm{d}x_{s_{n}}\mathrm{d}y_{s}\right|\nonumber \\
 & \leq & n!\left(1+\zeta\left(2\gamma\right)\right)^{n-1}\frac{\left\Vert \left(x,y\right)\right\Vert _{\gamma}^{n+1}}{\left(n+1\right)!^{\gamma}}\label{eq:geometric estimate}
\end{eqnarray}
where the inequality follows from Lyons' factorial decay estimate
(\ref{eq:lyons 94 estimate}). 

By Example \ref{extending a branched rough path}, we may extend the
two-dimensional path $\left(x,y\right)$ to a $\gamma-$branched rough
path $X$. Let $\boldsymbol{\sigma}_{n}$ be the labelled forest defined
inductively by $\boldsymbol{\sigma}_{1}=\bullet_{1}$ and $\boldsymbol{\sigma}_{n}=\boldsymbol{\sigma}_{n-1}\cdot\bullet_{1}$
where the operation $\cdot$ is the formal multiplication of rooted
trees. Let $\tau_{n}=\left[\boldsymbol{\sigma}_{n}\right]_{\bullet_{2}}$.
Graphically, $\tau_{n}$ takes the following form: 

\begin{figure}
\centering\includegraphics[scale=0.4]{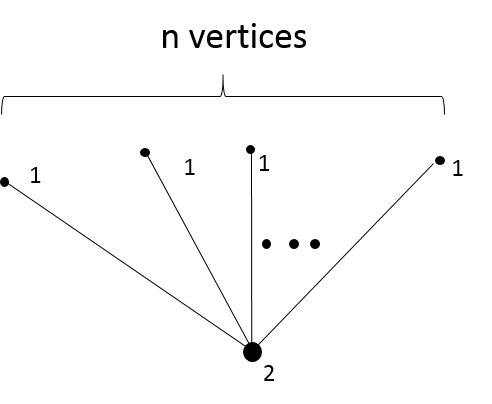}\caption{Tree with $n$ branches}
\end{figure}

Then 
\[
\langle X,\tau_{n}\rangle=\int_{0}^{1}x_{s}^{n}\mathrm{d}y_{s}
\]
 and our main result Theorem \ref{thm:main result} gives the estimate
\begin{eqnarray}
\left|\int_{0}^{1}x_{s}^{n}\mathrm{d}y_{s}\right| & \leq & \frac{\overline{c}_{1}^{n+1}}{\tau_{n}!^{\gamma}}=\frac{\overline{c}_{1}^{n+1}}{\left(n+1\right)^{\gamma}},\label{eq:branched estimate}
\end{eqnarray}
where $\overline{c}_{1}$ is the constant (independent of $n$) appearing
in our main result Theorem \ref{thm:main result}. As $n\rightarrow\infty$,
this growth rate is a much better rate than the factorial growth given
by the geometric estimate (\ref{eq:geometric estimate}). 

One reason why the geometric method fares badly here is that it bounds
the slowest decaying coordinate iterated integrals. This particular
integral 
\[
\int_{0}^{1}\int_{0}^{s}\ldots\int_{0}^{s_{2}}\mathrm{d}x_{s_{1}}\ldots\mathrm{d}x_{s_{n}}\mathrm{d}y_{s}
\]
decays a lot faster than the slowest decaying coordinate integral
of order $n+1$. The failure of the geometric method demonstrates
that the case of ``fat tree'' in Figure 1.1 is delicate and requires
a different type of estimate, which is the purpose of Section 5. This
problem of fat tree will be explored in more details in Section 3. 
\end{example}

\subsection{The strategy of proof}

Lyons \cite{Lyons98} proved the factorial decay for $\gamma$-geometric
rough paths using the following inductive definition of $X^{n}$ 
\begin{equation}
X_{u,t}^{n}=\lim_{\substack{|\mathcal{P}|\rightarrow0\\
\mathcal{P}\subset[u,t]
}
}\sum_{t_{i}\in\mathcal{P}}\sum_{k=1}^{n-1}X_{u,t_{i}}^{n-k}X_{t_{i},t_{i+1}}^{k}.\label{eq:riemann sum geometric}
\end{equation}
This approach requires the use of a highly non-trivial binomial-type
inequality, known as the neoclassical inequality, of the form 
\begin{equation}
\sum_{i=0}^{n}\frac{a^{i\gamma}b^{\left(n-i\right)\gamma}}{\Gamma\left(i\gamma+1\right)\Gamma\left(\left(n-i\right)\gamma+1\right)}\leq\gamma^{-2}\frac{\left(a+b\right)^{n\gamma}}{\Gamma\left(n\gamma+1\right)}\label{eq:original neoclassicla inequality}
\end{equation}
which is proved Lyons' 98 paper \cite{Lyons98}. A sharp version of
this inequality latter appeared in the work of Hara and Hino \cite{HaraHino10}.
Gubinelli showed in \cite{RamificationofBranchedRoughPath10} that
a sufficient condition for the factorial decay for branched rough
paths is a neoclassical inequality for rooted trees. Unfortunately,
we are able to give a counter example for such inequality (see Lemma
\ref{lem:Let--becounter example} in Section 3). Lyons' 94 approach
\cite{Lyons94}, which proved the factorial decay for the $\gamma>\frac{1}{2}$
case, did not use the neoclassical inequality and use instead the
equivalent definition 
\begin{equation}
X_{u,t}^{n}=\lim_{\substack{|\mathcal{P}|\rightarrow0\\
\mathcal{P}\subset[u,t]
}
}\sum_{t_{i}\in\mathcal{P}}\sum_{k=1}^{N}X_{u,t_{i}}^{n-k}X_{t_{i},t_{i+1}}^{k}\label{eq:riemann sum non geometric}
\end{equation}
where $N=\lfloor\gamma^{-1}\rfloor$. This approach also has its own
difficulty, due to the fact that the function $\left(s,t\right)\rightarrow\omega_{u}\left(s,t\right)$
defined by 
\[
\omega_{u}\left(s,t\right)=\left(\sum_{k=N+1}^{m}\frac{\left(s-u\right)^{m-k}\left(t-s\right)^{k}}{\left(m-k\right)!k!}\right)^{\frac{1}{N+1}}
\]
is not a control in the sense that $\omega_{u}\left(s,v\right)+\omega_{u}\left(v,t\right)\nleq\omega_{u}\left(s,t\right)$.
The control property is essential in the use of Young's method of
estimating (\ref{eq:riemann sum non geometric}) by successively removing
partition points from the partition. Lyons' 94 approach \cite{Lyons94}
gets around this problem by using a control function $\left(s,t\right)\rightarrow R_{u}\left(s,t\right)$
which dominates $\omega_{u}\left(s,t\right)$ and satisfies some binomial
properties similar to that of $\omega_{u}\left(s,t\right)$. A key
difficulty in this paper is to find the right function $R$ in the
case of branched rough paths. Our strategy consists of:
\begin{enumerate}
\item Proving a bound for the multiplication operator $\star$ with respect
to some norm, analogous to the following bound of tensor product 
\[
\left\Vert a\otimes b\right\Vert \leq\left\Vert a\right\Vert \left\Vert b\right\Vert 
\]
for $a\in V^{\otimes m}$ and $b\in V^{\otimes n}$ in the geometric
case. 
\item Prove that our function $R$ is compatible with the tree multiplication. 
\item Prove that our function $R$ is compatible with the operation of joining
forests to a single root. 
\end{enumerate}

\section{Branched Rough paths: Notation and Terminology }

We first recall the setting of the Connes-Kreimer \cite{Connes Kremier}
Hopf algebra which indexes Branched rough paths. A rooted tree is
a connected, rooted graph such that for every vertex in the graph,
there exists a unique path from the root to the vertex. Let $\mathcal{T}$
denote the set of rooted trees. The empty tree will be denoted by
$1$. A \textit{forest} is a finite set of rooted trees. The set of
forests will be denoted by $\mathcal{F}$. We will identify two trees
$\tau_{1}$ and $\tau_{2}$ if the forests obtained by removing the
respective roots from $\tau_{1}$ and $\tau_{2}$ are equal. We define
a commutative multiplication $\cdot$ on $\mathcal{F}$ by 
\[
x\cdot y=x\cup y.
\]
 Let $\bullet$ denote the rooted tree consisted of a single vertex.
We will use a bold symbol (e.g. $\boldsymbol{\tau}$) to denote a
forest while using the normal symbol (e.g. $\tau$) to denote a rooted
tree. For $\boldsymbol{\sigma}=\left\{ \tau_{1},\ldots,\tau_{n}\right\} \in\mathcal{F}$,
where $\tau_{1},\ldots,\tau_{n}$ are rooted non-empty trees, let
$\left[\boldsymbol{\sigma}\right]_{\bullet}$ denote the rooted tree
obtained by joining the roots of $\tau_{1},\ldots,\tau_{n}$ to the
vertex $\bullet$. Note that $\mathcal{F}$ is the set freely generated
by elements of the form $\left\{ \bullet\right\} $, through the operations
of $\cdot$ and $\boldsymbol{\sigma}\rightarrow\left[\boldsymbol{\sigma}\right]_{\bullet}$.
These two operations in fact correspond to the two fundamental operations
in rough path theory, namely the multiplication between path components
and the integration against a path.

To simplify our notation, we will denote the element $\left\{ \tau_{1},\ldots,\tau_{n}\right\} $
in $\mathcal{F}$ simply by $\tau_{1}\ldots\tau_{n}$. We will let
$\mathcal{H}$ denote the formal vector space spanned by $\mathcal{F}$
over $\mathbb{R}$. For a forest $\boldsymbol{\tau}$, $c\left(\boldsymbol{\tau}\right)$
will denote the number of non-empty trees in $\boldsymbol{\tau}$
and $\left|\boldsymbol{\tau}\right|$ denote the total number of vertices
in the forest. For each tree $\tau$, the tree factorial is defined
inductively as 
\begin{eqnarray*}
\bullet! & = & 1,\;\\
\left[\tau_{1},\ldots,\tau_{n}\right]_{\bullet}! & = & \left|\left[\tau_{1},\ldots,\tau_{n}\right]_{\bullet}\right|\tau_{1}!\ldots\tau_{n}!.
\end{eqnarray*}
The factorial of a forest $\tau_{1}\ldots\tau_{n}$ is defined to
be $\tau_{1}!\ldots\tau_{n}!$. 

A coproduct of rooted trees can be inductively defined as $\triangle:\mathcal{H}\rightarrow\mathcal{H}\otimes\mathcal{H}$,
\begin{eqnarray}
\triangle1 & = & 1\otimes1;\nonumber \\
\triangle\left[\tau_{1}\ldots\tau_{n}\right]_{\bullet} & = & \left[\tau_{1}\ldots\tau_{n}\right]_{\bullet}\otimes1+\sum\boldsymbol{\tau_{1}^{(1)}}\ldots\boldsymbol{\tau_{n}^{(1)}}\otimes\left[\tau_{1}^{(2)}\ldots\tau_{n}^{(2)}\right]_{\bullet};\label{eq:coproduct definition}\\
\triangle\left(\tau_{1}\ldots\tau_{n}\right) & = & \triangle\tau_{1}\ldots\triangle\tau_{n}.\nonumber 
\end{eqnarray}
where the sum in (\ref{eq:coproduct definition}) denotes summing
over all terms $\boldsymbol{\tau_{i}^{(1)}}$ and $\tau_{i}^{(2)}$
in $\triangle\tau_{i}=\sum\boldsymbol{\tau_{i}^{(1)}}\otimes\tau_{i}^{(2)}$.
While the coproduct was defined by Connes-Kreimer \cite{Connes Kremier},
this particular formulation was borrowed from Hairer-Kelly \cite{HairerKelly12}.
Here we define the product $\cdot$ on $\mathcal{H}\otimes\mathcal{H}$
by extending linearly the relation 
\[
\left(a\otimes b\right)\cdot\left(c\otimes d\right)=\left(a\cdot c\right)\otimes\left(b\cdot d\right).
\]
The coproduct operator $\triangle$ is coassociative. In Connes-Kreimer's
original work \cite{Connes Kremier}, an antipode operator $S$ has
been constructed explicitly for $\mathcal{H}$, so that the bialgebra
$\left(\mathcal{H},\cdot,\triangle,S\right)$ becomes a Hopf algebra.
This Hopf algebra is called the \textit{Connes-Kreimer Hopf algebra}. 
\begin{example}
The following are all the non-empty rooted trees with $3$ or less
vertices

\qquad{}\qquad{}\qquad{}\includegraphics[scale=0.4]{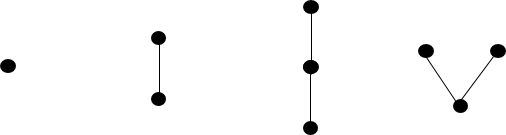}
\end{example}
The coproduct $\triangle$ also has an interpretation in terms of
cuts. A cut of a rooted tree is a set of edges in a rooted tree. A
cut is admissible for a rooted tree $\tau$ if for any vertex in $\tau$,
the path from the root to the vertex passes through at most one element
in the cut. For each admissible cuts $c$, let $\boldsymbol{\tau_{c}^{(1)}}$and
$\tau_{c}^{(2)}$denote, respectively the components in $\tau\backslash c$
that is disconnected from the the root and the component that is connected
to the root. Then 
\begin{equation}
\triangle\tau=\sum_{\mbox{Admissible cuts }c}\boldsymbol{\tau_{c}^{(1)}}\otimes\tau_{c}^{(2)}.\label{eq:coproduct expansion}
\end{equation}
Given a forest $\boldsymbol{\tau}=\tau_{1}\ldots\tau_{n}$ and $\boldsymbol{\sigma^{(1)}}$,$\boldsymbol{\sigma^{(2)}}$
in $\mathcal{F}$, we will define the counting function $c\left(\boldsymbol{\tau},\boldsymbol{\sigma^{(1)}},\boldsymbol{\sigma^{(2)}}\right)$
to be the number of times $\boldsymbol{\sigma^{(1)}}\otimes\boldsymbol{\sigma^{(2)}}$
appears in the sum (\ref{eq:coproduct expansion}). We will follow
the notation of Gubinelli \cite{RamificationofBranchedRoughPath10}
and use 
\[
\sum f\left(\boldsymbol{\tau},\boldsymbol{\tau^{(1)}},\boldsymbol{\tau^{(2)}}\right)
\]
to denote the summation over all $\boldsymbol{\tau^{(1)}}$ and $\boldsymbol{\tau^{(2)}}$
which appears in the sum (\ref{eq:coproduct expansion}).
\begin{rem}
Although the definition of branched rough paths requires the rooted
trees to be labelled, the assumption and conclusion of our main result
Theorem \ref{thm:main result} are uniform estimates across all labellings.
Therefore we can forget that there is any labelling and deal with
only unlabelled rooted trees or forests. We will let $\mathcal{T}^{n}$
and $\mathcal{F}^{n}$ denote, respectively, the set of all (unlabelled)
rooted trees and forests with $n$ vertices.
\end{rem}
We say $g\in\mathcal{H}^{*}$ lies in the group of characters of $\mathcal{H}^{*}$,
which we will denote by $\mathcal{G}$, if for all forests $\boldsymbol{\tau}$
and $\boldsymbol{\tilde{\tau}}$, 
\[
\langle g,\boldsymbol{\tau}\cdot\boldsymbol{\tilde{\tau}}\rangle=\langle X,\boldsymbol{\tau}\rangle\langle X,\boldsymbol{\tilde{\tau}}\rangle.
\]
In other words, $\mathcal{G}$ contains all the homomorphisms $g$
with respect to the tree multiplication $\cdot$. This formulation
can be found in, for instance Hairer-Kelly \cite{HairerKelly12}.

\section{Counter example to the tree neoclassical inequality }

We now give a counter example for a weaker version of the neoclassical
inequality, which would have been sufficient in proving the factorial
decay for the iterated integrals of branched rough paths. The notation
$\bullet^{n}$ will denote the forest 
\[
\underset{n}{\underbrace{\bullet\bullet\ldots\bullet}}.
\]

\begin{lem}
\label{lem:Let--becounter example}Let $\tau_{n}$ be the tree $\left[\bullet^{n}\right]_{\bullet}$.
Then for all $0\leq\gamma<1$, for all $\beta>0$, there exists $a,b>0$
such that as $n\rightarrow\infty$,  
\begin{equation}
\left(a+b\right)^{-\gamma|\tau_{n}|}\sum\left(\frac{\tau_{n}!}{\boldsymbol{\tau_{n}^{(1)}}!\tau_{n}^{(2)}!}\right)^{\gamma}\frac{1}{\beta^{c(\boldsymbol{\tau_{n}^{(1)}})+c(\tau_{n}^{(2)})}}a^{\gamma|\boldsymbol{\tau_{n}^{(1)}}|}b^{\gamma|\tau_{n}^{(2)}|}\rightarrow\infty.\label{eq:counter example.}
\end{equation}
\end{lem}
\begin{proof}
By definition, $\tau_{n}!=n+1$. Observe that by the definition of
coproduct $\triangle$ (see (\ref{eq:coproduct definition})), 
\[
\triangle\tau_{n}=\sum_{l=0}^{n}{n \choose l}\bullet^{l}\otimes\tau_{n-l}+\tau_{n}\otimes1,
\]
where ${n \choose l}$ denotes the binomial coefficient $\frac{n!}{\left(n-l\right)!l!}$.
Therefore, 
\begin{eqnarray}
 &  & \left(a+b\right)^{-\gamma|\tau_{n}|}\sum\Big(\frac{\tau_{n}!}{\boldsymbol{\tau_{n}^{(1)}}!\tau_{n}^{(2)}!}\Big)^{\gamma}\frac{1}{\beta^{c(\boldsymbol{\tau_{n}^{(1)}})+c(\tau_{n}^{(2)})}}a^{\gamma|\boldsymbol{\tau_{n}^{(1)}}|}b^{\gamma|\tau_{n}^{(2)}|}\nonumber \\
 & \geq & \left(a+b\right)^{-\gamma(n+1)}\sum_{l=0}^{n}\Big(\frac{n+1}{n+1-l}\Big)^{\gamma}{n \choose l}\frac{1}{\beta^{l+1}}a^{\gamma l}b^{\gamma(n+1-l)}\nonumber \\
 & \geq & \left(a+b\right)^{-\gamma\left(n+1\right)}b^{\gamma}\frac{1}{\beta}\left(\frac{a^{\gamma}}{\beta}+b^{\gamma}\right)^{n},\label{eq:divergent counter}
\end{eqnarray}
 where in the last line we used that for $l\leq n$, $\left(n+1\right)/\left(n+1-l\right)\geq1$
and the binomial theorem. Since $0\leq\gamma<1$ and $a\geq0$, 
\[
\left(1+a\right)^{\gamma}\leq1+\gamma a.
\]
Therefore, for all $0<a<\left(\beta\gamma\right)^{\frac{1}{\gamma-1}}$,
\[
\left(1+a\right)^{\gamma}<1+\frac{a^{\gamma}}{\beta}.
\]
Hence for $a<\left(\beta\gamma\right)^{\frac{1}{\gamma-1}}$ and $b=1$,
(\ref{eq:divergent counter}) diverges as $n$ tends to infinity. 
\end{proof}

\section{Bound for the multiplication operator $\star$}

The multiplication $\star$ in the Hopf algebra $\mathcal{H}$ plays
the role of the tensor product $\otimes$ in the theory of geometric
rough paths. In that case, one of the key assumptions made about the
tensor norms is that for all $X^{n}\in V^{\otimes n}$ and $Y^{k}\in V^{\otimes k}$,
\[
\left\Vert X^{n}\otimes Y^{k}\right\Vert _{V^{\otimes\left(n+k\right)}}\leq\left\Vert X^{n}\right\Vert _{V^{\otimes n}}\left\Vert Y^{k}\right\Vert _{V^{\otimes k}}
\]
so that the tensor multiplication has norm $1$. We might hope that
the multiplication with respect to $\star$ would also have norm $1$.
Unfortunately, given any numbers $n$ and $k$, rooted trees in general
has more than one way of being cut into two components of sizes $n$
and $k$ respectively. This causes the multiplication operation to
have a norm that potentially depends on $n$ and $k$. Fortunately,
and it is a key observation in our proof, the norm can be bounded
by a function of $k$, independently of $n$. Let us first describe
the norm that we use. 

Let $X\in\mathcal{H}^{*}$. Define a linear functional $X^{k}\in\mathcal{H}^{*}$
by 

\[
\left\langle X^{k},\tau\right\rangle =\begin{cases}
\left\langle X,\tau\right\rangle , & \left|\tau\right|=k;\\
0, & \left|\tau\right|\neq k.
\end{cases}
\]
We define $X^{n}\star Y^{k}$ such that for all forests $\boldsymbol{\tau}$,
\[
\left\langle X^{n}\star Y^{k},\boldsymbol{\tau}\right\rangle =\left\langle X^{n}\otimes Y^{k},\triangle\boldsymbol{\tau}\right\rangle .
\]
Let 
\[
\left\Vert X^{k}\right\Vert _{\mathcal{T},\gamma,\beta}=\max_{\left|\tau\right|=k,\tau\mbox{ trees}}\left|\left\langle X,\tau\right\rangle \right|\frac{\beta^{c\left(\tau\right)}\tau!^{\gamma}}{\left|\tau\right|!^{\gamma}},
\]
and 
\[
\left\Vert X^{k}\right\Vert _{\mathcal{F},\gamma,\beta}=\max_{\left|\tau\right|=k,\boldsymbol{\tau}\mbox{ forests}}\left|\left\langle X,\boldsymbol{\tau}\right\rangle \right|\frac{\beta^{c\left(\boldsymbol{\tau}\right)}\boldsymbol{\tau}!^{\gamma}}{\left|\boldsymbol{\tau}\right|!^{\gamma}}.
\]
In this section we will prove a bound on the norm of the multiplication
$\star$ with respect to $\left\Vert \cdot\right\Vert _{\mathcal{T},\gamma,\beta}$,
which is the first of three main steps in proving our main result. 
\begin{lem}
\label{lem:tree sum}(Multiplication is bounded in tree norm) Let
$\gamma\leq1$ and that 
\[
c_{k}:=\exp\big[\sum_{i=1}^{k}k{}^{i}(1-\gamma)\big],\;\beta\geq c_{k}.
\]
Let $X,Y\in\mathcal{H}^{*}$. Then for $n\geq1$, 
\[
\left\Vert X^{n}\star Y^{k}\right\Vert _{\mathcal{T},\gamma,\beta}\leq c_{k}\big|\mathcal{T}^{k}\big|^{1-\gamma}\beta^{-1}\left\Vert X^{n}\right\Vert _{\mathcal{F},\gamma,\beta}\left\Vert Y^{k}\right\Vert _{\mathcal{T},\gamma,\beta},
\]
where $\mathcal{T}^{k}$ denotes the set of rooted trees with $k$
vertices. 
\end{lem}
The proof will require a series of preliminary lemmas involving the
combinatorics of rooted trees. We will use $\sum_{\tau^{\left(2\right)}=\sigma}$
to denote the sum over all admissible cuts $c$ such that $\tau_{c}^{\left(2\right)}=\sigma$.
The following combinatorial lemma is crucial to proving our desired
lemma by induction. 
\begin{lem}
\label{lem:induction lemma}Let $\tau=\left[\tau_{1}\ldots\tau_{n}\right]_{\bullet}$,
where $\tau_{1},\ldots,\tau_{n}$ are non-empty. Let $\sigma=\left[\sigma_{1}\ldots\sigma_{n}\right]_{\bullet}\neq1$
be a rooted tree. Let $\sim_{\sigma}$ be a relation on the permutation
group $\mathbb{S}_{n}$ on $\{1,\ldots,n\}$ defined so that $\pi_{1}\sim\pi_{2}$
if $\sigma_{\pi_{1}(i)}=\sigma_{\pi_{2}(i)}$ for all $i$. Let $\mathbf{P}_{\sigma}$
be the set of $\sim_{\sigma}$-equivalent classes in the permutation
group $\mathbb{S}_{n}$. Then for all $\beta,\gamma>0$, 
\[
\sum_{\tau^{(2)}=\sigma}\frac{\beta^{-c(\boldsymbol{\tau^{(1)}})}}{\boldsymbol{\tau^{(1)}}!^{\gamma}}=\sum_{\pi\in\mathbf{P}_{\sigma}}\Pi_{i=1}^{n}\Big[\sum_{\tau_{i}^{(2)}=\sigma_{\pi(i)}}\frac{\beta^{-c(\boldsymbol{\tau_{i}^{(1)}})}}{\boldsymbol{\tau_{i}^{(1)}}!^{\gamma}}\Big].
\]
\end{lem}
\begin{proof}
Note that by the definition of $\triangle$, 
\begin{eqnarray}
\sum\boldsymbol{\tau^{\left(1\right)}}\otimes\tau^{\left(2\right)} & = & \tau\otimes1+\sum\boldsymbol{\tau_{1}^{\left(1\right)}}\ldots\boldsymbol{\tau_{n}^{\left(1\right)}}\otimes\left[\tau_{1}^{\left(2\right)}\ldots\tau_{n}^{\left(2\right)}\right]_{\bullet}.\label{eq:definition}
\end{eqnarray}
We define a linear functional $\boldsymbol{\sigma}$ such that for
each tree $a$, 
\begin{eqnarray*}
\boldsymbol{\sigma}\left(a\right) & = & 1,\;\mbox{if }a=\sigma,\\
 & = & 0,\;\mbox{if }a\neq\sigma.
\end{eqnarray*}
Let $\boldsymbol{f}$ be a linear functional defined such that for
each forest $\boldsymbol{\tau}$, 
\begin{eqnarray*}
\boldsymbol{f}\left(\boldsymbol{\tau}\right) & = & \frac{\beta^{c\left(\boldsymbol{\tau}\right)}}{\boldsymbol{\tau}!^{\gamma}}.
\end{eqnarray*}
Note that $f$ is a tree multiplication homomorphism. Define 
\[
\boldsymbol{f}\otimes\boldsymbol{\sigma}\left(a\otimes b\right)=\boldsymbol{f}\left(a\right)\boldsymbol{\sigma}\left(b\right).
\]
By applying $\boldsymbol{f}\otimes\boldsymbol{\sigma}$ to (\ref{eq:definition}),
\begin{eqnarray*}
\sum_{\tau^{(2)}=\sigma}\frac{\beta^{-c(\boldsymbol{\tau^{(1)}})}}{\boldsymbol{\tau^{(1)}}!^{\gamma}} & = & \sum_{[\tau_{1}^{(2)}\ldots\tau_{n}^{(2)}]{}_{\bullet}=\sigma}\frac{\beta^{-c(\boldsymbol{\tau_{1}^{(1)}})-\ldots-c(\boldsymbol{\tau_{n}^{(1)}})}}{\boldsymbol{\tau_{1}^{(1)}}!^{\gamma}\ldots\boldsymbol{\tau_{n}^{(1)}}!^{\gamma}}.
\end{eqnarray*}
As $[\tau_{1}^{(2)},\ldots,\tau_{n}^{(2)}]_{\bullet}=[\sigma_{1}^{(2)},\ldots,\sigma_{n}^{(2)}]_{\bullet}$
if and only if there exists $\pi\in\mathbf{P}_{\sigma}$ such that
$\tau_{i}^{(2)}=\sigma_{\pi\left(i\right)}$ for all $i$, 

\[
\sum_{\tau^{(2)}=\sigma}\frac{\beta^{-c(\boldsymbol{\tau^{(1)}})}}{\boldsymbol{\tau^{(1)}}!^{\gamma}}=\sum_{\pi\in\mathbf{P}_{\sigma}}\Pi_{i=1}^{n}\sum_{\tau_{i}^{(2)}=\sigma_{\pi(i)}}\frac{\beta^{-c(\boldsymbol{\tau_{i}^{(1)}})}}{\boldsymbol{\tau_{i}^{(1)}}!^{\gamma}}.
\]
\end{proof}
\begin{cor}
\label{cor:induction corollary}Let $\tau=\left[\tau_{1}\ldots\tau_{n}\right]_{\bullet}$,
where $\tau_{1},\ldots,\tau_{n}$ are non-empty. Let $\sigma=\left[\sigma_{1}\ldots\sigma_{n}\right]_{\bullet}\neq1$
be a rooted tree. Let $\mathbf{P}_{\tau,\sigma}^{\prime}$ denote
the set of all $\pi\in\mathbf{P}_{\sigma}$ such that $\sigma_{\pi(i)}\subseteq\tau_{i}$
for all $i$. Then for all $\beta,\gamma>0$,
\[
\sum_{\tau^{(2)}=\sigma}\frac{\beta^{-c(\boldsymbol{\tau^{(1)}})}}{\boldsymbol{\tau^{(1)}}!^{\gamma}}=\sum_{\pi\in\mathbf{P}_{\tau,\sigma}^{\prime}}\Pi_{i:\sigma_{\pi\left(i\right)}\subsetneq\tau_{i}}\Big[\sum_{\tau_{i}^{(2)}=\sigma_{\pi(i)}}\frac{\beta^{-c(\boldsymbol{\tau_{i}^{(1)}})}}{\boldsymbol{\tau_{i}^{(1)}}!^{\gamma}}\Big].
\]
\end{cor}
\begin{proof}
We have just shown in Lemma \ref{lem:induction lemma} that 
\begin{equation}
\sum_{\tau^{(2)}=\sigma}\frac{\beta^{-c(\boldsymbol{\tau^{(1)}})}}{\boldsymbol{\tau^{(1)}}!^{\gamma}}=\sum_{\pi\in\mathbf{P}_{\sigma}}\Pi_{i=1}^{n}\sum_{\tau_{i}^{(2)}=\sigma_{\pi(i)}}\frac{\beta^{-c(\boldsymbol{\tau_{i}^{(1)}})}}{\boldsymbol{\tau_{i}^{(1)}}!^{\gamma}}.\label{eq:induction relation 2-1}
\end{equation}
For $\pi\in\mathbf{P}_{\sigma}$, if an index $i$ is such that $\sigma_{\pi(i)}$
is not a subtree of $\tau_{i}$, then 
\[
\sum_{\tau_{i}^{(2)}=\sigma_{\pi(i)}}\frac{\beta^{-c(\boldsymbol{\tau_{i}^{(1)}})}}{\boldsymbol{\tau_{i}^{(1)}}!^{\gamma}}=0.
\]
Therefore, in (\ref{eq:induction relation 2-1}), summing over $\mathbf{P}_{\sigma}$
is equivalent to summing over $\mathbf{P}_{\tau,\sigma}^{\prime}$.
Furthermore, as $\tau_{i}^{\left(1\right)}=1$ if $\tau_{i}^{\left(2\right)}=\tau_{i}$,
\begin{eqnarray}
\sum_{\tau^{\left(2\right)}=\sigma}\frac{\beta^{-c(\boldsymbol{\tau^{(1)}})}}{\boldsymbol{\tau^{(1)}}!^{\gamma}} & = & \sum_{\pi\in\mathbf{P}_{\tau,\sigma}^{\prime}}\Pi_{i:\sigma_{\pi\left(i\right)}\subsetneq\tau_{i}}\Big[\sum_{\tau_{i}^{(2)}=\sigma_{\pi(i)}}\frac{\beta^{-c(\boldsymbol{\tau_{i}^{(1)}})}}{\boldsymbol{\tau_{i}^{\left(1\right)}}!^{\gamma}}\Big].\label{eq:equality-1}
\end{eqnarray}

\end{proof}
A key step in most factorial decay estimates for rough paths is to
take the fractional power $\gamma$ outside a sum. In the geometric
case, the job is done by the neoclassical inequality. We need the
following concavity estimate in the non-geometric case. 
\begin{lem}
\label{lem:Convexity estimate for trees}(Concavity estimate) Let
$\gamma\leq1$. For any rooted tree $\sigma$, let 
\begin{equation}
c_{|\sigma|}=\exp\Big[\sum_{i=1}^{|\sigma|}|\sigma|{}^{i}(1-\gamma)\Big]\;\mbox{and}\;\beta\geq c_{|\sigma|}.\label{eq:beta}
\end{equation}
For all rooted trees $\tau$ and $\sigma\subsetneq\tau$, we have
\begin{eqnarray*}
\sum_{\tau^{(2)}=\sigma}\frac{\beta^{-c(\boldsymbol{\tau^{(1)}})}}{\boldsymbol{\tau^{(1)}}!^{\gamma}} & \leq & c_{|\sigma|}\beta^{-1}\Big(\sum_{\tau^{(2)}=\sigma}\frac{1}{\boldsymbol{\tau^{(1)}}!}\Big)^{\gamma}.
\end{eqnarray*}
\end{lem}
\begin{rem}
The key point is that the constant we lose by taking the power $\gamma$
outside the sum, $c_{\left|\sigma\right|}$, depends only on $\left|\sigma\right|$
but not $\left|\tau\right|$. To achieve this, the conventional estimate
for sums 
\begin{equation}
\sum_{i=1}^{n}a_{i}^{\gamma}\leq n^{1-\gamma}\big(\sum_{i=1}^{n}a_{i}\big)^{\gamma}\label{eq:classicla concavity}
\end{equation}
is insufficient by itself. We must use the tree-multiplicative property
of the tree-factorial. 
\end{rem}
To prove the concavity estimate, Lemma \ref{lem:Convexity estimate for trees},
we first need a counting lemma. 
\begin{lem}
\label{lem:combin counting}Let $\sigma=[\sigma_{1},\ldots,\sigma_{n}]_{\bullet}$
and $\tau=[\tau_{1},\ldots,\tau_{n}]_{\bullet}$ be rooted trees such
that $\tau_{i}\neq1$ for all $i$. \textup{Let $k_{\tau,\sigma}=\min_{\pi\in\mathbf{P}_{\tau,\sigma}^{\prime}}\big|\{i:\sigma_{\pi\left(i\right)}\subsetneq\tau_{i}\}\big|$.
Then $k_{\tau,\sigma}\geq1$ and 
\[
|\mathbf{P}_{\tau,\sigma}^{\prime}|\leq\exp(|\sigma|^{2}k_{\tau,\sigma}).
\]
}\end{lem}
\begin{proof}
As $\tau\neq\sigma$ there does not exist permutation $\pi$ such
that $\sigma_{\pi(i)}=\tau_{i}$ for all $i$. In particular, we have
$k_{\tau,\sigma}\geq1$, which proves the first part of the lemma. 

Let $m$ be defined by $m=|\{i:\sigma_{i}=1\}|$. As $\sigma_{j}=\sigma_{l}$
for all $j,l\in\{i:\sigma_{i}=1\}$, each equivalence class in $\mathbf{P}_{\sigma}$
must contain at least $m!$ elements. Therefore, 
\[
|\mathbf{P}_{\tau,\sigma}^{\prime}|\leq|\mathbf{P}_{\sigma}|\leq\frac{n!}{m!}=n(n-1)\ldots(m+1).
\]
Since 
\[
n=m+\big|\{i:\sigma_{i}\neq1\}\big|\leq m+|\sigma|,
\]
we have 
\begin{eqnarray}
|\mathbf{P}_{\tau,\sigma}^{\prime}| & \leq & (m+|\sigma|)^{|\sigma|}.\label{eq:counting estimate}
\end{eqnarray}
Note that as $\tau_{i}$ is assumed to be non-empty for all $i$,
\begin{equation}
m=\big|\{i:\sigma_{i}=1\}\big|\leq\min_{\pi\in\mathbf{P}_{\tau,\sigma}^{\prime}}\big|\{i:\sigma_{\pi\left(i\right)}\subsetneq\tau_{i}\}\big|=k_{\tau,\sigma}.\label{eq:m and k}
\end{equation}
Using that for $x\geq1$ and $b\in\mathbb{N}\cup\{0\}$, 
\[
(x+b)^{b}\leq\exp\left(b^{2}x\right)
\]
in combination with the estimates (\ref{eq:counting estimate}) and
(\ref{eq:m and k}) earlier in this proof, 
\[
\left|\mathbf{P}_{\tau,\sigma}^{\prime}\right|\leq(m+|\sigma|)^{|\sigma|}\leq(k_{\tau,\sigma}+|\sigma|)^{|\sigma|}\leq\exp\left(|\sigma|^{2}k_{\tau,\sigma}\right).
\]

\end{proof}

\begin{proof}[Proof of concavity estimate Lemma 15]We will prove
the lemma by induction on $|\sigma|$. If $\left|\sigma\right|=0$,
then $\sigma=1$ and as $\sigma\neq\tau$, 
\begin{eqnarray*}
\sum_{\tau^{(2)}=\sigma}\frac{\beta^{-c(\boldsymbol{\tau^{(1)}})}}{\boldsymbol{\tau^{(1)}}!^{\gamma}} & = & \frac{\beta^{-1}}{\tau!^{\gamma}}
\end{eqnarray*}
which is exactly the content of the present lemma for the case $|\sigma|=0$.
Let $\tau=\left[\tau_{1}\ldots\tau_{n}\right]_{\bullet}$, where $\tau_{1},\ldots,\tau_{n}$
are all non-empty, and $\sigma=\left[\sigma_{1}\ldots\sigma_{n}\right]_{\bullet}$.
Using Corollary \ref{cor:induction corollary} that relates the sum
$\sum_{\tau^{(2)}=\sigma}$ to $\sum_{\tau_{i}^{(2)}=\sigma_{\pi\left(i\right)}}$,
\begin{eqnarray}
\sum_{\tau^{\left(2\right)}=\sigma}\frac{\beta^{-c(\boldsymbol{\tau^{(1)}})}}{\boldsymbol{\tau^{(1)}}!^{\gamma}} & = & \sum_{\pi\in\mathbf{P}_{\tau,\sigma}^{\prime}}\Pi_{i:\sigma_{\pi\left(i\right)}\subsetneq\tau_{i}}\Big[\sum_{\tau_{i}^{(2)}=\sigma_{\pi(i)}}\frac{\beta^{-c(\boldsymbol{\tau_{i}^{(1)}})}}{\boldsymbol{\tau_{i}^{(1)}}!^{\gamma}}\Big].\label{eq:equality}
\end{eqnarray}
By the induction hypothesis and that $k_{\tau,\sigma}=\min_{\pi\in\mathbf{P}_{\tau,\sigma}^{\prime}}\big|\{i:\sigma_{\pi\left(i\right)}\subsetneq\tau_{i}\}\big|$
by Lemma \ref{lem:combin counting}, 
\begin{eqnarray}
 &  & \sum_{\pi\in\mathbf{P}_{\tau,\sigma}^{\prime}}\Pi_{i:\sigma_{\pi\left(i\right)}\subsetneq\tau_{i}}\Big[\sum_{\tau_{i}^{(2)}=\sigma_{\pi(i)}}\frac{\beta^{-c(\boldsymbol{\tau_{i}^{(1)}})}}{\boldsymbol{\tau_{i}^{(1)}}!^{\gamma}}\Big]\label{eq:first step combin}\\
 & \leq & \Big(\beta^{-1}\exp\big(\sum_{j=1}^{|\sigma|-1}(|\sigma|-1)^{j}(1-\gamma)\big)\Big)^{k_{\tau,\sigma}}\label{eq:first step combin a}\\
 &  & \times\sum_{\pi\in\mathbf{P}_{\tau,\sigma}^{\prime}}\Big(\Pi_{i:\sigma_{\pi\left(i\right)}\subsetneq\tau_{i}}\sum_{\tau_{i}^{\left(2\right)}=\sigma_{\pi\left(i\right)}}\frac{1}{\boldsymbol{\tau_{i}^{\left(1\right)}}!}\Big)^{\gamma}.
\end{eqnarray}
By the conventional concavity estimate for sum $\sum_{i=1}^{n}a_{i}^{\gamma}\leq n^{1-\gamma}(\sum_{i=1}^{n}a_{i})^{\gamma}$,
\begin{eqnarray}
 &  & \sum_{\pi\in\mathbf{P}_{\tau,\sigma}^{\prime}}\Big(\Pi_{i:\sigma_{\pi\left(i\right)}\subsetneq\tau_{i}}\sum_{\tau_{i}^{(2)}=\sigma_{\pi(i)}}\frac{1}{\boldsymbol{\tau_{i}^{(1)}}!}\Big)^{\gamma}\label{eq:second step combin}\\
 & \leq & |\mathbf{P}_{\tau,\sigma}^{\prime}|^{1-\gamma}\Big(\sum_{\pi\in\mathbf{P}_{\tau,\sigma}^{\prime}}\Pi_{i:\sigma_{\pi\left(i\right)}\subsetneq\tau_{i}}\sum_{\tau_{i}^{(2)}=\sigma_{\pi(i)}}\frac{1}{\boldsymbol{\tau_{i}^{(1)}}!}\Big)^{\gamma}.
\end{eqnarray}
Using our estimate for $|\mathbf{P}_{\tau,\sigma}^{\prime}|$ in Lemma
\ref{lem:combin counting}, 
\begin{eqnarray}
 &  & \sum_{\pi\in\mathbf{P}_{\tau,\sigma}^{\prime}}\big(\Pi_{i:\sigma_{\pi\left(i\right)}\subsetneq\tau_{i}}\sum_{\tau_{i}^{(2)}=\sigma_{\pi(i)}}\frac{1}{\boldsymbol{\tau_{i}^{(1)}}!}\Big)^{\gamma}\label{eq:third step estimate}\\
 & \leq & \exp(|\sigma|^{2}k_{\tau,\sigma}(1-\gamma)\big)\Big(\sum_{\pi\in\mathbf{P}_{\tau,\sigma}^{\prime}}\Pi_{i:\sigma_{\pi\left(i\right)}\subsetneq\tau_{i}}\sum_{\tau_{i}^{(2)}=\sigma_{\pi(i)}}\frac{1}{\boldsymbol{\tau_{i}^{(1)}}!}\Big)^{\gamma}.
\end{eqnarray}
Combining the identity (\ref{eq:equality}) with all the inequalities
we have so far, namely (\ref{eq:first step combin a}) and (\ref{eq:third step estimate}), 

\begin{eqnarray*}
 &  & \sum_{\tau^{(2)}=\sigma}\frac{\beta^{-c(\boldsymbol{\tau^{(1)}})}}{\boldsymbol{\tau^{(1)}}!^{\gamma}}\\
 & \leq & \Big(\beta^{-1}\exp\big[(1-\gamma)(|\sigma|^{2}+\sum_{j=1}^{|\sigma|-1}(|\sigma|-1)^{j})\big]\Big)^{k_{\tau,\sigma}}\\
 &  & \times(\sum_{\pi\in\mathbf{P}_{\tau,\sigma}^{\prime}}\Pi_{i:\sigma_{\pi\left(i\right)}\subsetneq\tau_{i}}\sum_{\tau_{i}^{(2)}=\sigma_{\pi(i)}}\frac{1}{\boldsymbol{\tau_{i}^{(1)}}!}\big)^{\gamma}.
\end{eqnarray*}
Finally, as $k_{\tau,\sigma}\geq1$ (see counting lemma, Lemma \ref{lem:combin counting})
and $\beta\geq\exp\Big[\sum_{i=1}^{|\sigma|}|\sigma|{}^{i}(1-\gamma)\Big]$,
\[
\sum_{\tau^{(2)}=\sigma}\frac{\beta^{-c(\boldsymbol{\tau^{(1)}})}}{\boldsymbol{\tau^{(1)}}!^{\gamma}}\leq\beta^{-1}\exp\big(\sum_{j=1}^{|\sigma|}|\sigma|^{j}(1-\gamma)\big)(\sum_{\pi\in\mathbf{P}_{\tau,\sigma}^{\prime}}\Pi_{i:\sigma_{\pi\left(i\right)}\subsetneq\tau_{i}}\sum_{\tau_{i}^{(2)}=\sigma_{\pi(i)}}\frac{1}{\boldsymbol{\tau_{i}^{(1)}}!}\big)^{\gamma}.
\]
Note now by the $\gamma=1$ and $\beta=1$ case of Corollary \ref{cor:induction corollary},
\[
\sum_{\pi\in\mathbf{P}_{\tau,\sigma}^{\prime}}\Pi_{i:\sigma_{\pi\left(i\right)}\subsetneq\tau_{i}}\sum_{\tau_{i}^{(2)}=\sigma_{\pi(i)}}\frac{1}{\boldsymbol{\tau_{i}^{(1)}}!}=\sum_{\tau^{(2)}=\sigma}\frac{1}{\boldsymbol{\tau^{(1)}}!}.
\]
\end{proof}

We now state a lemma that is equivalent to Gubinelli's tree-binomial
theorem \cite{RamificationofBranchedRoughPath10}. It allows us to
rewrite sum over rooted trees to a sum over integers. 
\begin{lem}
\label{Tree-binomial-theorem}(Tree binomial theorem) Let $\tau$
be a rooted tree. Then

\[
\sum_{|\tau^{(2)}|=l}\frac{\tau!}{\boldsymbol{\tau^{(1)}}!\tau^{(2)}!}={|\tau| \choose l}.
\]
\end{lem}
\begin{proof}
By the tree binomial theorem, Lemma 4.4 in \cite{RamificationofBranchedRoughPath10},
we have for $x\in\mathbb{R}$,
\[
\frac{\left(1+x\right)^{|\tau|}}{\tau!}=\sum\frac{x^{|\tau^{(2)}|}}{\boldsymbol{\tau^{(1)}}!\tau^{(2)}!}.
\]
The result follows by comparing the coefficients of $x^{l}$ with
the classical binomial theorem. 
\end{proof}
We now prove the bound on the Hopf algebra multiplication $\star$.

\begin{proof}[Proof of boundednes of tree multiplication Lemma 11]Note
first that by the definition of $\star$, if $|\tau|=n+k$, 
\begin{eqnarray}
 &  & \left|\langle X^{n}\star Y^{k},\tau\rangle\right|\nonumber \\
 & = & \big|\sum_{|\tau^{(2)}|=k}\langle X^{n},\boldsymbol{\tau^{(1)}}\rangle\langle Y^{k},\tau^{(2)}\rangle\big|\\
 & \leq & \sum_{|\tau^{(2)}|=k}|\langle X^{n},\boldsymbol{\tau^{(1)}}\rangle|\,|\langle Y^{k},\tau^{(2)}\rangle|\nonumber 
\end{eqnarray}
By the definition of $\left\Vert \cdot\right\Vert _{\mathcal{F},\gamma,\beta}$
and $\left\Vert \cdot\right\Vert _{\mathcal{T},\gamma,\beta}$, 
\begin{eqnarray}
 &  & \sum_{|\tau^{(2)}|=k}|\langle X^{n},\boldsymbol{\tau^{(1)}}\rangle|\,|\langle Y^{k},\tau^{(2)}\rangle|\nonumber \\
 & \leq & \left\Vert X^{n}\right\Vert _{\mathcal{F},\gamma,\beta}\left\Vert Y^{k}\right\Vert _{\mathcal{T},\gamma,\beta}(n!k!)^{\gamma}\sum_{|\tau^{(2)}|=k}\frac{\beta^{-c(\boldsymbol{\tau^{(1)}})-1}}{(\boldsymbol{\tau^{(1)}}!\tau^{(2)}!)^{\gamma}}.\label{eq:briding to tree}
\end{eqnarray}
As we assumed that $n\geq1$, we have $k<\left|\tau\right|$ and hence
we may apply the concavity estimate for trees, Lemma \ref{lem:Convexity estimate for trees},
to obtain 
\begin{eqnarray*}
\sum_{|\tau^{(2)}|=k}\frac{\beta^{-c(\boldsymbol{\tau^{(1)}})}}{(\boldsymbol{\tau^{(1)}}!\tau^{(2)}!)^{\gamma}} & \leq & \sum_{|\sigma|=k}\frac{1}{\sigma!^{\gamma}}\sum_{\tau^{(2)}=\sigma}\frac{\beta^{-c(\boldsymbol{\tau^{(1)}})}}{\boldsymbol{\tau^{(1)}}!^{\gamma}}\\
 & \leq & c_{k}\beta^{-1}\sum_{|\sigma|=k}\frac{1}{\sigma!^{\gamma}}\big(\sum_{\tau^{(2)}=\sigma}\frac{1}{\boldsymbol{\tau^{(1)}}!}\big)^{\gamma}\\
 & \leq & c_{k}\big|\mathcal{T}^{k}\big|^{1-\gamma}\beta^{-1}\big(\sum_{|\tau^{(2)}|=k}\frac{1}{\boldsymbol{\tau^{(1)}}!\tau^{(2)}!}\big)^{\gamma}.
\end{eqnarray*}
We now use the tree binomial theorem (Lemma \ref{Tree-binomial-theorem})
to deduce that 
\begin{equation}
\sum_{|\tau^{(2)}|=k}\frac{\beta^{-c(\boldsymbol{\tau^{(1)}})}}{\left(\boldsymbol{\tau^{(1)}}!\tau^{(2)}!\right)^{\gamma}}\leq c_{k}\big|\mathcal{T}^{k}\big|^{1-\gamma}\beta^{-1}\frac{1}{\tau!^{\gamma}}{|\tau| \choose k}^{\gamma}.\label{eq:bridging 2}
\end{equation}
Therefore, for all rooted trees $\tau$ such that $\left|\tau\right|=n+k$,
by substituting (\ref{eq:bridging 2}) into (\ref{eq:briding to tree}),
\[
\left|\left\langle X^{n}\star Y^{k},\tau\right\rangle \right|\leq c_{k}\big|\mathcal{T}^{k}\big|^{1-\gamma}\big(\frac{\left|\tau\right|!}{\tau!}\big)^{\gamma}\beta^{-2}\left\Vert X^{n}\right\Vert _{\mathcal{F},\gamma,\beta}\left\Vert Y^{k}\right\Vert _{\mathcal{T},\gamma,\beta}
\]
and we have 
\begin{eqnarray*}
\left\Vert X^{n}\star Y^{k}\right\Vert _{\mathcal{T},\gamma,\beta} & = & \max_{|\tau|=n+k,\tau\mbox{ trees}}\beta\left|\left\langle X^{n}\star Y^{k},\tau\right\rangle \right|\frac{\tau!{}^{\gamma}}{|\tau|!^{\gamma}}\\
 & \leq & c_{k}\big|\mathcal{T}^{k}\big|^{1-\gamma}\beta^{-1}\left\Vert X^{n}\right\Vert _{\mathcal{F},\gamma,\beta}\left\Vert Y^{k}\right\Vert _{\mathcal{T},\gamma,\beta}.
\end{eqnarray*}
\end{proof}

\section{Compatibility of our estimate with tree multiplication}

We showed in the last section the multiplicative bound 
\begin{equation}
\left\Vert X^{n}\star Y^{k}\right\Vert _{\mathcal{T},\gamma,\beta}\leq c_{k}\big|\mathcal{T}^{k}\big|^{1-\gamma}\beta^{-1}\left\Vert X^{n}\right\Vert _{\mathcal{F},\gamma,\beta}\left\Vert Y^{k}\right\Vert _{\mathcal{T},\gamma,\beta}\label{eq:multiplicative bound}
\end{equation}
with $\beta\geq c_{k}$. That we can choose a large $\beta$ is very
useful. It will help us to annihilate any constant depending on $k$.
Suppose that $X_{\cdot,\cdot}$ is a branched rough path and $X_{s,t}^{n}$
denotes the restriction of the branched rough path $X$ on trees with
$n$ vertices. Let $(t_{0},\ldots,t_{r})$ be a partition for $[s,t]$,
then in a similar spirit to Lyons 94'\cite{Lyons94}, we have 
\begin{equation}
X_{s,t}^{n+1}=\lim_{|t_{i}-t_{i+1}|\rightarrow\infty}\sum_{i=0}^{r-1}\sum_{k=1}^{\lfloor\frac{1}{\gamma}\rfloor}X_{s,t_{i}}^{n+1-k}\star X_{t_{i},t_{i+1}}^{k}.\label{eq:Lyons 94 formula}
\end{equation}
We would like to apply the multiplicative bound (\ref{eq:multiplicative bound})
to estimate the Riemann sum 
\begin{equation}
\left\Vert \sum_{i=0}^{r-1}\sum_{k=1}^{\lfloor\frac{1}{\gamma}\rfloor}X_{s,t_{i}}^{n+1-k}\star X_{t_{i},t_{i+1}}^{k}\right\Vert _{\mathcal{T},\gamma,\beta}.\label{eq:Riemann sum}
\end{equation}
A crucial point is that the biggest $k$ can be here is $\lfloor\frac{1}{\gamma}\rfloor$
which is independent of $n$. The constant in front of $\left\Vert X^{n}\right\Vert _{\mathcal{F},\gamma,\beta}\left\Vert Y^{k}\right\Vert _{\mathcal{T},\gamma,\beta}$
in the multiplicative bound (\ref{eq:multiplicative bound}) is therefore
independent of $n$. If we had use the following formula instead (as
in Lyons 98'\cite{Lyons98}) 
\begin{equation}
X_{s,t}^{n+1}=\lim_{|t_{i}-t_{i+1}|\rightarrow0}\sum_{i=0}^{r-1}\sum_{k=1}^{n}X_{s,t_{i}}^{n+1-k}\star X_{t_{i},t_{i+1}}^{k},\label{eq:Lyons 98 formula}
\end{equation}
then the biggest $k$ can be is $n$, and the constant in (\ref{eq:multiplicative bound})
would depend on $n$. This is the fundamental reason why we must use
the approach in Lyons 94' (\ref{eq:Lyons 94 formula}) instead of
using (\ref{eq:Lyons 98 formula}) as in Lyons 98'. 

After applying the bound (\ref{eq:multiplicative bound}) to estimate
the Riemann sum (\ref{eq:Riemann sum}), we should have a bound for
the tree norm 
\[
\left\Vert X_{s,t}^{n+1}\right\Vert _{\mathcal{T},\gamma,\beta}.
\]
However, to use the multiplicative bound (\ref{eq:multiplicative bound})
for estimating the tree norm $\left\Vert X_{s,t}^{n+2}\right\Vert _{\mathcal{T},\gamma,\beta}$,
it is not enough to know only the tree norm $\left\Vert X_{s,t}^{n+1}\right\Vert _{\mathcal{T},\gamma,\beta}$.
We must also know the forest norm $\left\Vert X_{s,t}^{n+1}\right\Vert _{\mathcal{F},\gamma,\beta}$.
That is why we must find a way of bounding the forest norm $\left\Vert X_{s,t}^{n+1}\right\Vert _{\mathcal{F},\gamma,\beta}$
in terms of $\left\Vert X_{s,t}^{n+1}\right\Vert _{\mathcal{T},\gamma,\hat{\beta}}$
(with $\hat{\beta}>\beta$), which we aim to achieve in this section
through a lemma. We first present the form of our estimate. 

Let $\triangle_{m}(r,r^{\prime})$ denote the $m$-dimensional simplex
\[
\left\{ \left(s_{1},\ldots,s_{m}\right)\in\mathbb{R}^{m}:r<s_{1}<\ldots<s_{m}<r^{\prime}\right\} .
\]
For a one-dimensional path $\rho$, we will define 
\[
S^{\left(m\right)}\left(\rho\right)_{s,t}=\int_{\triangle_{m}\left(s,t\right)}\mathrm{d}\rho\left(s_{1}\right)\ldots\mathrm{d}\rho\left(s_{m}\right).
\]
For each $a,b>0$ define a one-dimensional path $\rho$ by 
\[
\rho_{a}^{b}\left(t\right)=\frac{1}{b}\left(t-a\right)^{b}.
\]
The construction of our estimate is based on the following lemma,
which says that our estimate dominates the tail of a binomial sum.
Its proof can be found in the Appendix. 
\begin{lem}
\label{Taylor binomial}Let $N\in\mathbb{N}\cup\left\{ 0\right\} $
and $n\geq N+1$. For all $u<s<t$ 
\begin{eqnarray}
\sum_{j=N+1}^{n}\frac{\left(s-u\right)^{n-j}\left(t-s\right)^{j}}{\left(n-j\right)!j!} & \leq & \frac{1}{\left(n-N-1\right)!}S^{\left(N+1\right)}\left(\rho_{u}^{\frac{n}{N+1}}\right)_{s,t}.\label{eq:dominate binomial sum}
\end{eqnarray}

\end{lem}
The following lemma is the main result of this section and allows
us to convert our bound from one about the tree norm to the forest
norm. 
\begin{lem}
\label{tree to forest}(Compatibility with tree multiplication) Let
$0<\gamma\leq1$ and $N=\lfloor\gamma^{-1}\rfloor$. Let $X$ be a
$\gamma$-branched rough path. Let 
\[
\hat{c}_{N}=3|\mathcal{T}^{N}|^{1-\gamma}(N+1)^{3\left(1-\gamma\right)}\exp2\left(N+1\right),\;\beta\geq\hat{c}_{N}.
\]
 Suppose that for all $n\leq M$ and $u\leq s\leq t$, 
\begin{eqnarray}
\Vert\sum_{k\geq N+1}X_{u,s}^{n-k}\star X_{s,t}^{k}\Vert_{\mathcal{T},\gamma,\beta} & \leq & \left[\frac{1}{\left(n-N-1\right)!}S^{\left(N+1\right)}\left(\rho_{u}^{\frac{n}{N+1}}\right)_{s,t}\right]^{\gamma}.\label{eq:tree estimate}
\end{eqnarray}
Then for all $n\leq M$, 
\begin{eqnarray*}
\Vert\sum_{k\geq N+1}X_{u,s}^{n-k}\star X_{s,t}^{k}\Vert_{\mathcal{F},\gamma,\beta\hat{c}_{N}^{-1}} & \leq & \left[\frac{1}{\left(n-N-1\right)!}S^{\left(N+1\right)}\left(\rho_{u}^{\frac{n}{N+1}}\right)_{s,t}\right]^{\gamma}.
\end{eqnarray*}

\end{lem}
We will once again need a series of lemmas. The first of which states
that for factorial decay estimates the forest norm $\Vert\cdot\Vert_{\mathcal{F},\gamma,\beta}$
is the same as tree norm $\Vert\cdot\Vert_{\mathcal{T},\gamma,\beta}$. 
\begin{lem}
\label{lem:factorial decay estimate}Let $X\in\mathcal{G}$. Let $k\geq0$
and 
\[
\beta\geq\exp\big[\sum_{i=1}^{k}k{}^{i}(1-\gamma)\big].
\]
If there exists $a>0$ such that 
\begin{equation}
\left\Vert X^{n}\right\Vert _{\mathcal{T},\gamma,\beta}\leq\frac{a^{\gamma n}}{n!^{\gamma}},\label{eq:tree factorial estimate}
\end{equation}
then 
\begin{equation}
\left\Vert X^{n}\right\Vert _{\mathcal{F},\gamma,\beta}\leq\frac{a^{\gamma n}}{n!^{\gamma}}.\label{eq:forest factorial decay estimate}
\end{equation}
\end{lem}
\begin{proof}
By the assumption (\ref{eq:tree factorial estimate}), for all rooted
trees $\tau$ such that $|\tau|=n$, 
\[
|\langle X^{n},\tau\rangle|\leq\frac{a^{\gamma n}}{\beta\tau!^{\gamma}}.
\]
Therefore, for any forest $\boldsymbol{\tau}=\tau_{1}\ldots\tau_{m}$,
where $|\boldsymbol{\tau}|=n$ and $\tau_{1}\ldots\tau_{m}$ are rooted
trees, by the 
\begin{eqnarray*}
|\langle X^{n},\boldsymbol{\tau}\rangle| & = & |\langle X,\boldsymbol{\tau}\rangle|\\
 & = & |\langle X,\tau_{1}\rangle|\ldots|\langle X,\tau_{m}\rangle|\\
 & = & |\langle X^{|\tau_{1}|},\tau_{1}\rangle|\ldots|\langle X^{|\tau_{m}|},\tau_{m}\rangle|\\
 & \leq & \frac{a^{\gamma n}}{\beta^{m}\boldsymbol{\tau}!^{\gamma}},
\end{eqnarray*}
which is equivalent to our desired factorial decay estimate (\ref{eq:forest factorial decay estimate}). 
\end{proof}
To extend estimates about rooted trees to estimates about forests,
we usually need to carry out induction on the number of components
in the forest. To carry out such induction, the following algebraic
identity is very useful. 
\begin{lem}
\label{lem:algebraic lemma for forest}(Forest factorisation lemma)
Let $X,Y\in\mathcal{G}$. Then for any forests $\boldsymbol{\tau}$
and $\boldsymbol{\tilde{\tau}}$ such that $n+k=|\boldsymbol{\tau}|+|\boldsymbol{\tilde{\tau}}|$,
\[
\langle X^{n}\star Y^{k},\boldsymbol{\tau}\boldsymbol{\tilde{\tau}}\rangle=\sum_{k_{1}+k_{2}=k}\langle X^{|\boldsymbol{\tau}|-k_{1}}\star Y^{k_{1}},\boldsymbol{\tau}\rangle\langle X^{|\boldsymbol{\tilde{\tau}}|-k_{2}}\star Y^{k_{2}},\boldsymbol{\tilde{\tau}}\rangle.
\]
\end{lem}
\begin{proof}
By definition of $\star$ and that the coproduct $\triangle$ is compatible
with tree multiplication, 
\begin{eqnarray}
\langle X^{n}\star Y^{k},\boldsymbol{\tau}\boldsymbol{\tilde{\tau}}\rangle & = & \langle X^{n}\otimes Y^{k},\triangle(\boldsymbol{\tau}\boldsymbol{\tilde{\tau}})\rangle\nonumber \\
 & = & \langle X^{n}\otimes Y^{k},\triangle\boldsymbol{\tau}\triangle\boldsymbol{\tilde{\tau}}\rangle\nonumber \\
 & = & \sum_{|\boldsymbol{\tau}^{(2)}|+|\boldsymbol{\tilde{\tau}}^{(2)}|=k}\langle X^{n},\boldsymbol{\tau}^{(1)}\boldsymbol{\tilde{\tau}}^{(1)}\rangle\langle Y^{k},\boldsymbol{\tau}^{(2)}\boldsymbol{\tilde{\tau}}^{(2)}\rangle.\label{eq:coprodcut computation 1}
\end{eqnarray}
As $n+k=|\boldsymbol{\tau}|+|\boldsymbol{\tilde{\tau}}|$ and that
we are summing over $|\boldsymbol{\tau}^{(2)}|+|\boldsymbol{\tilde{\tau}}^{(2)}|=k$,
we have in particular that $n=|\boldsymbol{\tau}^{(1)}|+|\boldsymbol{\tilde{\tau}}^{(1)}|$
in the sum. As $X\in\mathcal{H}^{*}$, 
\begin{eqnarray*}
\langle X^{n},\boldsymbol{\tau}^{(1)}\boldsymbol{\tilde{\tau}}^{(1)}\rangle & = & \langle X,\boldsymbol{\tau}^{(1)}\boldsymbol{\tilde{\tau}}^{(1)}\rangle\\
 & = & \langle X,\boldsymbol{\tau}^{(1)}\rangle\langle X,\boldsymbol{\tilde{\tau}}^{(1)}\rangle\\
 & = & \langle X^{|\boldsymbol{\tau}^{(1)}|},\boldsymbol{\tau}^{(1)}\rangle\langle X^{|\tilde{\boldsymbol{\tau}}^{(1)}|},\boldsymbol{\tilde{\tau}}^{(1)}\rangle.
\end{eqnarray*}
Analogous expression also holds for $\langle Y^{k},\boldsymbol{\tau}^{(2)}\boldsymbol{\tilde{\tau}}^{(2)}\rangle$
with the same proof. Therefore, by our earlier calculation (\ref{eq:coprodcut computation 1})
and the definition of $\star$, 
\begin{eqnarray*}
 &  & \langle X^{n}\star Y^{k},\boldsymbol{\tau}\boldsymbol{\tilde{\tau}}\rangle\\
 & = & \sum_{|\boldsymbol{\tau}^{(2)}|+|\boldsymbol{\tilde{\tau}}^{(2)}|=k}\langle X^{|\boldsymbol{\tau}^{(1)}|},\boldsymbol{\tau}^{(1)}\rangle\langle X^{|\tilde{\boldsymbol{\tau}}^{(1)}|},\boldsymbol{\tilde{\tau}}^{(1)}\rangle\langle Y^{|\boldsymbol{\tau}^{(2)}|},\boldsymbol{\tau}^{(2)}\rangle\langle Y^{|\tilde{\boldsymbol{\tau}}^{(2)}|},\boldsymbol{\tilde{\tau}}^{(2)}\rangle\\
 & = & \sum_{k_{1}+k_{2}=k}\langle X^{|\boldsymbol{\tau}|-k_{1}}\star Y^{k_{1}},\boldsymbol{\tau}\rangle\langle X^{|\boldsymbol{\tilde{\tau}}|-k_{2}}\star Y^{k_{2}},\boldsymbol{\tilde{\tau}}\rangle.
\end{eqnarray*}

\end{proof}
The following lemma states that if we assume $X^{n}$ and $Y^{k}$
have factorial decay estimates, then the forest norm of $X^{n}\star Y^{k}$
can be bounded by the tree norms of $X^{n}$ and $Y^{k}$. 
\begin{lem}
\label{lem:Computation lemma-1}(Multiplication is bounded in forest
norm) Let $X,Y\in\mathcal{H}^{*}$. Let $k\geq0$ and 
\[
\beta\geq\exp\big[\sum_{i=1}^{k}k{}^{i}(1-\gamma)\big].
\]
If there exists $a>0$ and $b>0$ such that 
\[
\left\Vert X^{n}\right\Vert _{\mathcal{T},\gamma,\beta}\leq\frac{a^{\gamma n}}{n!^{\gamma}},\;\left\Vert Y^{k}\right\Vert _{\mathcal{T},\gamma,\beta}\leq\frac{b^{\gamma k}}{k!^{\gamma}},
\]
then 
\[
\left\Vert X^{n}\star Y^{k}\right\Vert _{\mathcal{F},\gamma,\tilde{c}_{k}^{-1}\beta}\leq\frac{a^{\gamma n}b^{\gamma k}}{\left(n!k!\right)^{\gamma}},
\]
where $\tilde{c}_{k}=c_{k}((k+1)\left|\mathcal{T}_{k}\right|){}^{1-\gamma}$.\end{lem}
\begin{proof}
We need to show that for all forests $\boldsymbol{\tau}$, and $\left(n,k\right)$
such that $n+k=|\boldsymbol{\tau}|$, 
\begin{equation}
\left|\langle X^{n}\star Y^{k},\boldsymbol{\tau}\rangle\right|\leq\frac{\tilde{c}_{k}^{c(\boldsymbol{\tau})}a^{\gamma n}b^{\gamma k}}{\beta^{c(\boldsymbol{\tau})}\boldsymbol{\tau}!^{\gamma}}{|\boldsymbol{\tau}| \choose k}^{\gamma}.\label{eq:factorial desired estimate}
\end{equation}
We shall prove it by induction on $c(\boldsymbol{\tau})$. If $n=0$,
then present lemma directly follows from assumption and we will henceforth
assume $n\geq1$. For $c(\boldsymbol{\tau})=1$, note that in this
case $\boldsymbol{\tau}$ is a tree. By the boundedness of group multiplication
in tree norm, Lemma \ref{lem:tree sum}, and that the tree norm of
$X$ is the same as the forest norm of $X$ (see Lemma \ref{lem:factorial decay estimate}),
\[
\left\Vert X^{n}\star Y^{k}\right\Vert _{\mathcal{T},\gamma,\beta}\leq c_{k}\big|\mathcal{T}^{k}\big|^{1-\gamma}\beta^{-1}\frac{a^{\gamma n}b^{\gamma k}}{n!^{\gamma}k!^{\gamma}},
\]
which implies our desired estimate (\ref{eq:factorial desired estimate})
in the case when $\boldsymbol{\tau}$ is a rooted tree. For the induction
step, let $\boldsymbol{\tau}=\tau_{1}\boldsymbol{\tau}_{2}$, where
$\tau_{1}$ is a non-empty tree and $\boldsymbol{\tau}_{2}$ is a
forest. If $n+k=\left|\boldsymbol{\tau}\right|$, then by the forest
factorisation lemma, Lemma \ref{lem:algebraic lemma for forest},
and the induction hypothesis, 
\begin{eqnarray*}
K & := & \left|\langle X^{n}\star Y^{k},\tau_{1}\boldsymbol{\tau}_{2}\rangle\right|\\
 & \leq & \sum_{l+m=k}|\langle X^{|\tau_{1}|-l}\star Y^{l},\tau_{1}\rangle|\,|\langle X^{|\boldsymbol{\tau}_{2}|-m}\star Y^{m},\boldsymbol{\tau}_{2}\rangle|\\
 & \leq & \frac{\tilde{c}_{k}{}^{c(\boldsymbol{\tau}_{2})}c_{k}\big|\mathcal{T}^{k}|^{1-\gamma}}{\beta^{c(\boldsymbol{\tau}_{2})+1}\left(\tau_{1}!\boldsymbol{\tau}_{2}!\right)^{\gamma}}a^{\gamma n}b^{\gamma k}\sum_{l+m=k}{|\tau_{1}| \choose l}^{\gamma}{|\boldsymbol{\tau}_{2}| \choose m}^{\gamma}.
\end{eqnarray*}
Using the conventional concavity estimate for sum $\sum_{i=1}^{M}a_{i}^{\gamma}\leq M^{1-\gamma}(\sum_{i=1}^{M}a_{i})^{\gamma}$
and that $\tilde{c}_{k}=c_{k}((k+1)|\mathcal{T}^{k}|)^{1-\gamma}$,

\begin{eqnarray*}
K & \leq & \frac{\tilde{c}_{k}{}^{c(\boldsymbol{\tau}_{2})}c_{k}\big|\mathcal{T}^{k}|^{1-\gamma}}{\beta^{c(\boldsymbol{\tau}_{2})+1}\left(\tau_{1}!\boldsymbol{\tau}_{2}!\right)^{\gamma}}a^{\gamma n}b^{\gamma k}(k+1)^{1-\gamma}\big(\sum_{l+m=k}{|\tau_{1}| \choose l}{|\boldsymbol{\tau}_{2}| \choose m}\big)^{\gamma}\\
 & \leq & \frac{\tilde{c}_{k}{}^{(c(\boldsymbol{\tau}_{2})+1)}}{\beta^{c(\boldsymbol{\tau}_{2})+1}\left(\tau_{1}!\boldsymbol{\tau}_{2}!\right)^{\gamma}}{|\tau_{1}|+|\boldsymbol{\tau}_{2}| \choose k}^{\gamma}a^{\gamma n}b^{\gamma k},
\end{eqnarray*}
which in particular implies our desired estimate (\ref{eq:factorial desired estimate}).
\end{proof}
We will also need the following binomial lemma that describes the
product of our estimate over $\left[s,t\right]$ with the length of
the overlapping interval $\left[u,t\right]$, where $u\leq s\leq t$. 
\begin{lem}
\label{lem:binomial lemma for overlapping intervals}(Binomial lemma
for overlapping time intervals) Let $N\in\mathbb{N}\cup\left\{ 0\right\} $.
Let $n\geq N+1$ and $m\geq0$, and $u\leq s\leq t$, 
\begin{eqnarray}
 &  & \frac{n!}{\left(n-N-1\right)!}S^{\left(N+1\right)}\left(\rho_{u}^{\frac{n}{N+1}}\right)_{s,t}\left(t-u\right)^{m}\label{eq:tree multiply}\\
 & \leq & \frac{\left(n+m\right)!}{\left(n+m-N-1\right)!}S^{\left(N+1\right)}\left(\rho_{u}^{\frac{n+m}{N+1}}\right)_{s,t}.
\end{eqnarray}

\end{lem}

The proof of Lemma \ref{lem:binomial lemma for overlapping intervals}
can be found in the Appendix.

\begin{proof}[Proof of the compatibility with tree multiplication, Lemma 20]We
need to show that for all forests $\boldsymbol{\tau}$, 
\begin{eqnarray*}
\big|\langle\sum_{k\geq N+1}X_{u,s}^{n-k}\star X^{k},\boldsymbol{\tau}\rangle\big| & \leq & \big(\frac{\hat{c}_{N}}{\beta}\big)^{c\left(\boldsymbol{\tau}\right)}\big[\frac{|\boldsymbol{\tau}|!}{\boldsymbol{\tau}!\left(|\boldsymbol{\tau}|-N-1\right)!}S^{\left(N+1\right)}\left(\rho_{u}^{\frac{|\boldsymbol{\tau}|}{N+1}}\right)_{s,t}\big]^{\gamma}
\end{eqnarray*}
and will do so via induction on $c\left(\boldsymbol{\tau}\right)$.
The case $c\left(\boldsymbol{\tau}\right)=1$ follows directly from
the assumption. 

By putting $u=s$ in the assumed estimate (\ref{eq:tree estimate})
for trees and using $\left(N+1\right)^{N+1}\left[\left(N+1\right)!\right]^{-1}\leq\exp\left(N+1\right)$,
we have that for all forests $\boldsymbol{\tau}$ with $|\boldsymbol{\tau}|\leq M$,
\begin{equation}
\left|\left\langle X_{s,t},\boldsymbol{\tau}\right\rangle \right|\leq\frac{C_{N}^{c\left(\boldsymbol{\tau}\right)}\left(t-s\right)^{\gamma\left|\boldsymbol{\tau}\right|}}{\beta^{c\left(\boldsymbol{\tau}\right)}\boldsymbol{\tau}!^{\gamma}}\label{eq:deduction}
\end{equation}
where $C_{N}=\exp\left(N+1\right)$ and hence for all $n\leq M$ and
all $s\leq t$, 
\[
\Vert X_{s,t}^{n}\Vert_{\mathcal{F},\gamma,\beta C_{N}^{-1}}\leq\frac{\left(t-s\right)^{\gamma n}}{n!^{\gamma}}.
\]
The estimate (\ref{eq:deduction}) in particular says that our estimate
is a factorial decay estimate.

For the induction step, we let $\tau_{1}$ be a non-empty rooted tree
and $\boldsymbol{\tau}_{2}$ be a forest such that $\boldsymbol{\tau}=\tau_{1}\boldsymbol{\tau}_{2}$.
By the forest factorisation lemma, Lemma \ref{lem:algebraic lemma for forest},
\begin{eqnarray}
 &  & \langle\sum_{k\geq N+1}X_{u,s}^{n-k}\star X_{s,t}^{k},\tau_{1}\boldsymbol{\tau}_{2}\rangle\nonumber \\
 & = & \sum_{l+m\geq N+1}\langle X_{u,s}^{|\tau_{1}|-l}\star X_{s,t}^{l},\tau_{1}\rangle\langle X_{u,s}^{|\boldsymbol{\tau}_{2}|-m}\star X_{s,t}^{m},\boldsymbol{\tau}_{2}\rangle\nonumber \\
 & = & \big(\sum_{l\geq N+1}+\sum_{\substack{l\leq N\\
m\geq N+1
}
}+\sum_{\substack{l,m\leq N\\
N+1-l\leq m
}
}\big)\langle X_{u,s}^{|\tau_{1}|-l}\star X_{s,t}^{l},\tau_{1}\rangle\langle X_{u,s}^{|\boldsymbol{\tau}_{2}|-m}\star X_{s,t}^{m},\boldsymbol{\tau}_{2}\rangle.\label{eq:breaking up forest}
\end{eqnarray}
We will denote the three terms in this decomposition (\ref{eq:breaking up forest})
as $K_{1},K_{2}$ and $K_{3}$ respectively. Using the assumed estimate
(\ref{eq:tree estimate}) for trees and factorial decay estimate (\ref{eq:deduction}),
\begin{eqnarray}
K_{1} & := & \big|\sum_{l\geq N+1}\langle X_{u,s}^{|\tau_{1}|-l}\star X_{s,t}^{l},\tau_{1}\rangle\langle X_{u,s}^{|\boldsymbol{\tau}_{2}|-m}\star X_{s,t}^{m},\boldsymbol{\tau}_{2}\rangle\big|.\label{eq:first step tree to forest}\\
 & \leq & \frac{C_{N}^{c(\boldsymbol{\tau}_{2})}}{\beta^{c(\boldsymbol{\tau})}}\Big[\frac{|\tau_{1}|!}{\tau_{1}!\left(|\tau_{1}|-N-1\right)!}S^{\left(N+1\right)}\Big(\rho_{u}^{\frac{|\tau_{1}|}{N+1}}\Big)_{s,t}\frac{\left(t-u\right)^{|\boldsymbol{\tau}_{2}|}}{\boldsymbol{\tau_{2}}!}\Big]^{\gamma}.
\end{eqnarray}
Using the binomial lemma for overlapping intervals, Lemma \ref{lem:binomial lemma for overlapping intervals},
\begin{equation}
K_{1}\leq\frac{C_{N}^{c\left(\boldsymbol{\tau}_{2}\right)}}{\beta^{c\left(\boldsymbol{\tau}\right)}}\Big[\frac{|\boldsymbol{\tau}|!}{\boldsymbol{\tau}!\left(|\boldsymbol{\tau}|-N-1\right)!}S^{\left(N+1\right)}\left(\rho_{u}^{\frac{\left|\boldsymbol{\tau}\right|}{N+1}}\right)_{s,t}\Big]^{\gamma}.\label{eq:K1 estime}
\end{equation}
We now estimate the second term in the decomposition (\ref{eq:breaking up forest}).
Using that multiplication is bounded in tree norm (see Lemma \ref{lem:Computation lemma-1},
applicable as $\beta\geq\exp\big(\sum_{i=1}^{N}N^{i}$$\big)$) and
the factorial decay estimate (\ref{eq:deduction}),

\begin{eqnarray}
\big|\sum_{l\leq N}\langle X_{u,s}^{|\tau_{1}|-l}\star X_{s,t}^{l},\tau_{1}\rangle\big| & \leq & \frac{\tilde{c}_{N}C_{N}^{2}}{\beta\tau_{1}!^{\gamma}}\sum_{l\leq N}{|\tau_{1}| \choose l}^{\gamma}\left(s-u\right)^{\gamma\left(|\tau_{1}|-l\right)}\left(t-s\right)^{\gamma l}\nonumber \\
 & \leq & \frac{\tilde{c}_{N}C_{N}^{2}(N+1)^{1-\gamma}}{\beta\tau_{1}!^{\gamma}}\Big(\sum_{l\leq N}{|\tau_{1}| \choose l}\left(s-u\right)^{\left(|\tau_{1}|-l\right)}\left(t-s\right)^{l}\Big)^{\gamma}\\
 & \leq & \frac{\tilde{c}_{N}C_{N}^{2}(N+1)^{1-\gamma}}{\beta\tau_{1}!^{\gamma}}\left(t-u\right)^{\gamma|\tau_{1}|}.\label{eq:int cal 3}
\end{eqnarray}
Applying the induction hypothesis and (\ref{eq:int cal 3}),
\begin{eqnarray*}
K_{2} & := & \big|\sum_{l\leq N,m\geq N+1}\langle X_{u,s}^{|\tau_{1}|-l}\star X_{s,t}^{l},\tau_{1}\rangle\langle X_{u,s}^{|\boldsymbol{\tau}_{2}|-m}\star X_{s,t}^{m},\boldsymbol{\tau}_{2}\rangle\big|.\\
 & \leq & \frac{\tilde{c}_{N}C_{N}^{2}(N+1)^{1-\gamma}\hat{c}_{N}^{c(\boldsymbol{\tau}_{2})}}{\beta^{c(\boldsymbol{\tau})}}\Big[\frac{|\boldsymbol{\tau}_{2}|!\left(t-u\right)^{|\tau_{1}|}}{\boldsymbol{\tau}_{2}!\left(|\boldsymbol{\tau}_{2}|-N-1\right)!\tau_{1}!}S^{\left(N+1\right)}\left(\rho_{u}^{\frac{|\boldsymbol{\tau}_{2}|}{N+1}}\right)_{s,t}\Big]^{\gamma}.
\end{eqnarray*}
 By the binomial lemma for overlapping intervals (\ref{eq:tree multiply})
\begin{eqnarray}
 &  & K_{2}\leq\frac{\tilde{c}_{N}C_{N}^{2}(N+1)^{1-\gamma}\hat{c}_{N}^{c(\boldsymbol{\tau}_{2})}}{\beta^{c(\boldsymbol{\tau})}}\Big[\frac{\left|\boldsymbol{\tau}\right|!}{\boldsymbol{\tau}!\left(\left|\boldsymbol{\tau}\right|-N-1\right)!}S^{\left(N+1\right)}\left(\rho_{u}^{\frac{\left|\boldsymbol{\tau}\right|}{N+1}}\right)_{s,t}\Big]^{\gamma}.\label{eq:int cal 5}
\end{eqnarray}
Finally, we estimate the third term in the decomposition (\ref{eq:breaking up forest})
at the beginning of this proof. By factorial decay estimate (\ref{eq:deduction})
and applying Lemma \ref{lem:Computation lemma-1}, which asserts that
the multiplication is bounded in forest norm for factorial decay estimates,

\begin{eqnarray*}
K_{3} & := & \big|\sum_{\substack{l,m\leq N\\
N+1-l\leq m
}
}\langle X_{u,s}^{|\tau_{1}|-l}\star X_{s,t}^{l},\tau_{1}\rangle\langle X_{u,s}^{|\boldsymbol{\tau}_{2}|-m}\star X_{s,t}^{m},\boldsymbol{\tau}_{2}\rangle\big|\\
 & \leq & \sum_{\substack{l,m\leq N\\
N+1-l\leq m
}
}\big(\frac{c_{1,N}}{\beta}\big)^{c\left(\boldsymbol{\tau}\right)}\big({|\tau_{1}| \choose l}{|\boldsymbol{\tau}_{2}| \choose m}\frac{\left(s-u\right)^{|\boldsymbol{\tau}|-l-m}\left(t-s\right)^{l+m}}{\tau_{1}!\boldsymbol{\tau}_{2}!}\big)^{\gamma}
\end{eqnarray*}
where $c_{1,N}=\tilde{c}_{N}C_{N}^{2}$. By the conventional concavity
estimate for sum $\sum_{i=1}^{M}a_{i}^{\gamma}\leq M^{1-\gamma}(\sum_{i=1}^{m}a_{i})^{\gamma}$,
\[
K_{3}\leq\left(\frac{c_{2,N}}{\beta}\right)^{c\left(\boldsymbol{\tau}\right)}(\sum_{\substack{l,m\leq N\\
N+1-l\leq m
}
}{|\tau_{1}| \choose l}{|\boldsymbol{\tau}_{2}| \choose m}\frac{\left(s-u\right)^{|\boldsymbol{\tau}|-l-m}\left(t-s\right)^{l+m}}{\tau_{1}!\boldsymbol{\tau}_{2}!}\big)^{\gamma},
\]
where $c_{2,N}=(N+1)^{2\left(1-\gamma\right)}\tilde{c}_{N}C_{N}^{2}$.
Using the binomial identity $\sum_{l+m=k}{M_{1} \choose l}{M_{2} \choose m}=\sum{M_{1}+M_{2} \choose k}$,
\[
K_{3}\leq\left(\frac{c_{2,N}}{\beta}\right)^{c\left(\boldsymbol{\tau}\right)}\big(\sum_{N+1\leq k}\frac{\left(s-u\right)^{\left|\boldsymbol{\tau}\right|-k}\left(t-s\right)^{k}}{\tau_{1}!\boldsymbol{\tau}_{2}!}{|\boldsymbol{\tau}| \choose k}\big)^{\gamma}.
\]
As our estimate dominates the tail of the binomial sum (see Lemma
\ref{Taylor binomial}),
\begin{equation}
K_{3}\leq\left(\frac{c_{2,N}}{\beta}\right)^{c\left(\boldsymbol{\tau}\right)}\left[\frac{\left|\boldsymbol{\tau}\right|!}{\boldsymbol{\tau}!\left(\left|\boldsymbol{\tau}\right|-N-1\right)!}S^{\left(N+1\right)}\left(\rho_{u}^{\frac{\left|\boldsymbol{\tau}\right|}{N+1}}\right)_{s,t}\right]^{\gamma}.\label{eq:small sum}
\end{equation}
 Therefore, substituting the estimates for $K_{1}$ (\ref{eq:K1 estime}),
$K_{2}$ (\ref{eq:int cal 5}) and $K_{3}$ (\ref{eq:small sum})
into the decomposition (\ref{eq:breaking up forest}), we have 
\begin{eqnarray*}
 &  & \big|\langle\sum_{k\geq N+1}X_{u,s}^{|\boldsymbol{\tau}|-k}\star X_{s,t}^{k},\boldsymbol{\tau}\rangle\big|\\
 & \leq & \beta^{-c(\boldsymbol{\tau})}\big(C_{N}^{c(\boldsymbol{\tau})}+\big((N+1)^{2(1-\gamma)}\tilde{c}_{N}C_{N}^{2}\big)^{c(\boldsymbol{\tau})}+\tilde{c}_{N}C_{N}^{2}(N+1)^{1-\gamma}\hat{c}_{N}^{c(\boldsymbol{\tau})-1}\big)\\
 &  & \times\big(\frac{\left|\boldsymbol{\tau}\right|!}{\boldsymbol{\tau}!\left(\left|\boldsymbol{\tau}\right|-N-1\right)!}S^{\left(N+1\right)}\left(\rho_{u}^{\frac{\left|\boldsymbol{\tau}\right|}{N+1}}\right)_{s,t}\big)^{\gamma}
\end{eqnarray*}
and the Lemma follows by $3\tilde{c}_{N}C_{N}^{2}(N+1)^{2\left(1-\gamma\right)}\leq\hat{c}_{N}$.
\end{proof}

\section{The proof}

Let $X$ be a $\gamma-$branched rough path and let $N=\lfloor\gamma^{-1}\rfloor$.
We will use the following identity that is implicit in Gubinelli's
construction of iterated integrals of branched rough path (see Theorem
7.3 in \cite{RamificationofBranchedRoughPath10}) 
\begin{equation}
X_{s,t}^{n+1}=\lim_{\left|\mathcal{P}\right|\rightarrow0}\sum_{i=0}^{m-1}\sum_{k=1}^{N}X_{s,t_{i}}^{n+1-k}\star X_{t_{i},t_{i+1}}^{k},\label{eq:induction relation}
\end{equation}
for $n\geq N+1$, where the limit is taken as the mesh size $\max_{t_{i}\in\mathcal{P}}\left|t_{i+1}-t_{i}\right|$
of the partition 
\[
\mathcal{P}=\left(0=t_{0}<\ldots<t_{m}=1\right)
\]
goes to zero. Alternatively, one can check directly that the limit
on the right hand side of (\ref{eq:induction relation}) converges,
has the multiplicative property and is $\gamma$-Hölder, which by
Theorem 7.3 in \cite{RamificationofBranchedRoughPath10} would imply
(\ref{eq:induction relation}). We will estimate the double sum on
the right hand side in (\ref{eq:induction relation}) by dropping
points successively from the partition $\mathcal{P}$. The following
lemma carries out the algebra of removing partition points from a
Riemann sum.
\begin{lem}
\label{drop point algebra}Let $X$ be a $\gamma-$branched rough
path and let $N=\lfloor\gamma^{-1}\rfloor$. For each partition $\mathcal{P}$
of the interval $\left[s,t\right]$, define $X^{\mathcal{P},n}:\left[0,1\right]\times\left[0,1\right]\rightarrow\mathcal{H}^{*}$
such that
\end{lem}
\[
X_{s,t}^{\mathcal{P},n}=\sum_{t_{i}\in\mathcal{P}}\sum_{1\leq k\leq N}X_{s,t_{i}}^{n-k}\star X_{t_{i},t_{i+1}}^{k}
\]
Then for any $t_{j}\in\mathcal{P}$, 
\begin{eqnarray}
 &  & \sum_{k\geq N+1}X_{u,s}^{n-k}\star\big(X_{s,t}^{\mathcal{P},k}-X_{s,t}^{\mathcal{P}\backslash\{t_{j}\},k}\big)\label{eq:drop points}\\
 & = & \sum_{\substack{k_{2}+k_{3}\geq N+1\\
1\leq k_{3}\leq N
}
}X_{u,t_{j-1}}^{n-k_{2}-k_{3}}\star X_{t_{j-1},t_{j}}^{k_{2}}\star X_{t_{j},t_{j+1}}^{k_{3}}.
\end{eqnarray}

\begin{proof}
Note first that 
\begin{eqnarray*}
X_{s,t}^{\mathcal{P},k}-X_{s,t}^{\mathcal{P}\backslash\{t_{j}\},k} & = & \sum_{1\leq l\leq N}X_{s,t_{j-1}}^{k-l}\star X_{t_{j-1},t_{j}}^{l}+X_{s,t_{j}}^{k-l}\star X_{t_{j},t_{j+1}}^{l}-X_{s,t_{j-1}}^{k-l}\star X_{t_{j-1},t_{j+1}}^{l}.
\end{eqnarray*}
By applying the multiplicativity of $X$ to the third term, 
\begin{eqnarray}
X_{s,t}^{\mathcal{P},k}-X_{s,t}^{\mathcal{P}\backslash\{t_{j}\},k} & = & \sum_{\substack{1\leq l_{3}\leq N\\
l_{2}\geq N+1-l_{3}
}
}X_{s,t_{j-1}}^{k-l_{2}-l_{3}}\star X_{t_{j-1},t_{j}}^{l_{2}}\star X_{t_{j},t_{j+1}}^{l_{3}}.\label{eq:substract a point}
\end{eqnarray}
From this, we observe that $X_{s,t}^{\mathcal{P},k}-X_{s,t}^{\mathcal{P}\backslash\{t_{j}\},k}$
is nonzero only when $k\geq N+1$. Therefore, 
\begin{eqnarray}
\sum_{k\geq N+1}X_{u,s}^{n-k}\star\big(X_{s,t}^{\mathcal{P},k}-X_{s,t}^{\mathcal{P}\backslash\{t_{j}\},k}\big) & = & \sum_{k}X_{u,s}^{n-k}\star\big(X_{s,t}^{\mathcal{P},k}-X_{s,t}^{\mathcal{P}\backslash\{t_{j}\},k}\big).\label{eq:after cancellation}
\end{eqnarray}
By substituting (\ref{eq:substract a point}) into (\ref{eq:after cancellation})
and applying the associativity of $\star$, we see that 
\[
\sum_{k\geq N+1}X_{u,s}^{n-k}\star\big(X_{s,t}^{\mathcal{P},k}-X_{s,t}^{\mathcal{P}\backslash\{t_{j}\},k}\big)=\sum_{\substack{1\leq l_{3}\leq N\\
l_{2}\geq N+1-l_{3}
}
}X_{u,t_{j-1}}^{n-l_{2}-l_{3}}\star X_{t_{j-1},t_{j}}^{l_{2}}X_{t_{j},t_{j+1}}^{l_{3}}.
\]

\end{proof}
We now once again require some binomial-type lemmas which we will
prove in the Appendix. The following says that our estimate is decreasing
in some sense.
\begin{lem}
\label{lem:For-all-decreasing}For all $0\leq k\leq m\leq n$ and
$u\leq s\leq t$, 
\begin{equation}
\frac{1}{\left(n-m\right)!}S^{\left(m\right)}\left(\rho_{u}^{\frac{n}{m}}\right)_{s,t}\leq\frac{\exp m}{\left(n-m+k\right)!}S^{\left(m-k\right)}\left(\rho_{u}^{\frac{n}{m-k}}\right)_{s,t}.\label{eq:decreasing}
\end{equation}

\end{lem}
The following is at the heart of the proof of our main result. It
describes the product of our estimates over adjacent time intervals. 
\begin{lem}
\label{lem:binomial lemma for adjacent intervals}(Binomial lemma
for adjacent intervals) Let $0\leq k\leq m\leq n$ and $u\leq s\leq t\leq v$,
then 
\begin{eqnarray}
S^{\left(m-k\right)}\left(\rho_{u}^{\frac{n}{m-k}}\right)_{s,t}\frac{\left(v-t\right)^{k}}{k!} & \leq & S^{\left(m-k\right)}\left(\rho_{u}^{\frac{n+k}{m}}\right)_{s,t}S^{\left(k\right)}\left(\rho_{u}^{\frac{n+k}{m}}\right)_{t,v}.\label{eq:transfer of power-1}
\end{eqnarray}
\end{lem}
\begin{proof}
In the third line below, we used that $s_{j}>s_{i}$ for $m-k+1\leq j\leq m$
and $1\leq i\leq m-k$, 
\begin{eqnarray*}
 &  & S^{\left(m-k\right)}\left(\frac{m-k}{n}\left(\cdot-u\right)^{\frac{n}{m-k}}\right)_{s,t}\frac{\left(v-t\right)^{k}}{k!}\\
 & = & \int_{s<s_{1}<\ldots<s_{m-k}<t}\Pi_{i=1}^{m-k}\left(s_{i}-u\right)^{\frac{n+k-m}{m-k}}\mathrm{d}s_{1}\ldots\mathrm{d}s_{m-k}\\
 &  & \times\int_{t<s_{m-k+1}<\ldots<s_{m}<v}\mathrm{d}s_{m-k+1}\ldots\mathrm{d}s_{m}\\
 & \leq & \int_{s<s_{1}<\ldots<s_{m-k}<t}\Pi_{i=1}^{m-k}\left(s_{i}-u\right)^{\frac{n+k-m}{m}}\mathrm{d}s_{1}\ldots\mathrm{d}s_{m-k}\\
 &  & \times\int_{t<s_{m-k+1}<\ldots<s_{m}<v}\Pi_{i=m-k+1}^{m}\left(s_{i}-u\right)^{\frac{n+k-m}{m}}\mathrm{d}s_{m-k+1}\ldots\mathrm{d}s_{m}\\
 & = & S^{\left(m-k\right)}\left(\frac{m}{n+k}\left(\cdot-u\right)^{\frac{n+k}{m}}\right)_{s,t}S^{\left(k\right)}\left(\frac{m}{n+k}\left(\cdot-u\right)^{\frac{n+k}{m}}\right)_{t,v}.
\end{eqnarray*}

\end{proof}
The following result gives an estimate for the remainder of a coproduct
sum of branched rough paths. It may look like we are proving more
than the factorial decay result we need, but in fact such estimate
provides exactly the necessary induction hypothesis to prove the factorial
decay estimate. 
\begin{lem}
\label{main lemma}Let $0<\gamma\leq1$ and $N=\lfloor\gamma^{-1}\rfloor$
. Let $X$ be a $\gamma-$branched rough path. If for any $0\leq n\leq N$,
\begin{eqnarray}
\Vert X_{s,t}^{n}\Vert_{\mathcal{T},\gamma,\beta} & \leq & \frac{\left(t-s\right)^{n\gamma}}{n!^{\gamma}},\label{eq:decay of derivatives}
\end{eqnarray}
and 
\begin{equation}
\beta\geq6\exp\left(7\sum_{i=1}^{N+1}\left(N+1\right)^{i}\right)\sum_{r=2}^{\infty}\left(\frac{2}{r-1}\wedge1\right)^{\left(N+1\right)\gamma}|\mathcal{T}_{N}|{}^{1-\gamma}\label{eq:beta condition}
\end{equation}
then the following holds for all $n$,
\begin{eqnarray}
 &  & \Vert\sum_{k\geq N+1}X_{u,s}^{n-k}\star X_{s,t}^{k}\Vert_{\mathcal{T},\gamma,\beta}\leq\left[\frac{1}{\left(n-N-1\right)!}S^{\left(N+1\right)}\left(\rho_{u}^{\frac{n}{N+1}}\right)_{s,t}\right]^{\gamma}1_{\left\{ n\geq N+1\right\} }.\label{eq:final conclusion}
\end{eqnarray}
\end{lem}
\begin{proof}
We shall prove this by induction on $n$. The base induction $n=N$
is trivial as both sides in equation (\ref{eq:final conclusion})
is zero. By the induction hypothesis and the compatibility of our
estimate with tree multiplication (Lemma \ref{tree to forest}), for
all $m\leq n$, 
\begin{eqnarray*}
 &  & \Vert\sum_{l\geq N+1}X_{u,s}^{m-l}\star X_{s,t}^{l}\Vert_{\mathcal{F},\gamma,\hat{c}_{N}^{-1}\beta}\leq\left[\frac{1}{\left(m-N-1\right)!}S^{\left(N+1\right)}\left(\rho_{u}^{\frac{m}{N+1}}\right)_{s,t}\right]^{\gamma}1_{\left\{ m\geq N+1\right\} },
\end{eqnarray*}
and as our estimate is decreasing (see Lemma \ref{lem:For-all-decreasing}),
\begin{eqnarray*}
 &  & \Vert\sum_{l\geq N+1}X_{u,s}^{m-l}\star X_{s,t}^{l}\Vert_{\mathcal{F},\gamma,\hat{c}_{N}^{-1}\beta}\leq\left[\frac{\exp\left(N+1\right)}{\left(m-r\right)!}S^{\left(r\right)}\left(\rho_{u}^{\frac{m}{r}}\right)_{s,t}\right]^{\gamma}1_{\left\{ m\geq N+1\right\} }.
\end{eqnarray*}
for any $1\leq r\leq N$. As multiplication is bounded in forest norm
(see Lemma \ref{lem:Computation lemma-1}) for factorial decay estimates,
we have for all $1\leq r\leq N$, 
\begin{eqnarray*}
K^{\prime} & := & \Vert\sum_{r\leq l\leq N}X_{u,s}^{m-l}\star X_{s,t}^{l}\Vert_{\mathcal{F},\gamma,\left(\tilde{c}_{N}C_{N}\right)^{-1}\beta}\\
 & \leq & \sum_{r\leq l\leq N}\Vert X_{u,s}^{m-l}\Vert_{\mathcal{T},\gamma,\left(\tilde{c}_{N}C_{N}\right)^{-1}\beta}\Vert X_{s,t}^{l}\Vert_{\mathcal{T},\gamma,\left(\tilde{c}_{N}C_{N}\right)^{-1}\beta}.
\end{eqnarray*}
As our estimate is a factorial decay estimate, 
\[
K^{\prime}\leq\sum_{r\leq l\leq N}\frac{\left(s-u\right)^{\gamma(m-l)}(t-s)^{\gamma l}}{(m-l)!^{\gamma}l!^{\gamma}}.
\]
By the conventional concavity estimate for sum, 
\[
K^{\prime}\leq(N+1)^{1-\gamma}\big(\sum_{r\leq l\leq N}\frac{\left(s-u\right)^{(m-l)}(t-s)^{l}}{(m-l)!l!}\big)^{\gamma}.
\]
By the binomial Lemma which bounds the remainder of a binomial sum
by our estimate (Lemma \ref{Taylor binomial}), 
\[
K^{\prime}\leq(N+1)^{1-\gamma}\Big[\frac{1}{\left(m-r\right)!}S^{\left(r\right)}\left(\rho_{u}^{\frac{m}{n-r}}\right)_{s,t}\Big]^{\gamma}.
\]
In particular, since $\hat{c}_{N}\geq\tilde{c}_{N}C_{N}$, for all
$r\leq N$, 
\begin{eqnarray}
\Vert\sum_{r\leq l}X_{u,s}^{m-l}\star X_{s,t}^{l}\Vert_{\mathcal{F},\gamma,\hat{c}_{N}^{-1}\beta} & \leq & c_{5,N}\left[\frac{1}{\left(m-r\right)!}S^{\left(r\right)}\left(\rho_{u}^{\frac{m}{r}}\right)_{s,t}\right]^{\gamma},\label{eq:deduction 2}
\end{eqnarray}
where $c_{5,N}=2\exp(N+1)$. Note first that as multiplication $\star$
is bounded in tree norm (Lemma \ref{lem:tree sum}), 

\begin{eqnarray*}
 &  & \Vert\sum_{\substack{1\leq k\leq N\\
l\geq N+1-k
}
}X_{u,t_{j-1}}^{n-l-k}\star X_{t_{j-1},t_{j}}^{l}\star X_{t_{j},t_{j+1}}^{k}\Vert_{\mathcal{T},\gamma,\hat{c}_{N}^{-1}\beta}\\
 & \leq & c_{N}\beta^{-1}\sum_{1\leq k\leq N}\Vert\sum_{l\geq N+1-k}X_{u,t_{j-1}}^{n-l-k}\star X_{t_{j-1},t_{j}}^{l}\Vert_{\mathcal{F},\gamma,\hat{c}_{N}^{-1}\beta}\Vert X_{t_{j},t_{j+1}}^{k}\Vert_{\mathcal{T},\gamma,\hat{c}_{N}^{-1}\beta}.
\end{eqnarray*}
By our assumption that we have a factorial decay estimate for $X^{1},\ldots,X^{N}$
(see (\ref{eq:decay of derivatives})) and (\ref{eq:deduction 2}), 

\begin{eqnarray*}
\tilde{K} & := & \Vert\sum_{l\geq N+1-k}X_{u,t_{j-1}}^{n-l-k}\star X_{t_{j-1},t_{j}}^{k}\Vert_{\mathcal{F},\gamma,\hat{c}_{N}^{-1}\beta}\Vert X_{t_{j},t_{j+1}}^{k}\Vert_{\mathcal{T},\gamma,\hat{c}_{N}^{-1}\beta}\\
 & \leq & c_{5,N}\big[\frac{1}{(n-N-1)!}S^{\left(N+1-k\right)}\bigg(\rho_{u}^{\frac{n-k}{N+1-k}}\bigg)_{t_{j-1},t_{j}}\frac{\left(t_{j+1}-t_{j}\right)^{k}}{k!}\big]^{\gamma}.
\end{eqnarray*}
By the binomial Lemma for adjacent intervals, Lemma \ref{lem:binomial lemma for adjacent intervals}
, 
\[
\tilde{K}\leq c_{5,N}\big[\frac{1}{(n-N-1)!}S^{\left(N+1-k\right)}\left(\rho_{u}^{\frac{n}{N+1}}\right)_{t_{j-1},t_{j}}S^{\left(k\right)}\left(\rho_{u}^{\frac{n}{N+1}}\right)_{t_{j},t_{j+1}}\big]^{\gamma}.
\]
From here we use the classical concavity estimate for sums and Chen's
identity (see for example Theorem 2.1.2 in \cite{Lyons98}) to obtain
\begin{eqnarray}
 &  & \sum_{1\leq k\leq N}\Vert\sum_{l\geq N+1-k}X_{u,t_{j-1}}^{n-l-k}\star X_{t_{j-1},t_{j}}^{l}\Vert_{\mathcal{F},\gamma,\hat{c}_{N}^{-1}\beta}\Vert X_{t_{j},t_{j+1}}^{k}\Vert_{\mathcal{T},\gamma,\hat{c}_{N}^{-1}\beta}\nonumber \\
 & \leq & \frac{c_{6,N}}{\beta}\big[\frac{1}{(n-N-1)!}\sum_{k_{3}=1}^{N}S^{(N+1-k)}\big(\rho_{u}^{\frac{n}{N+1}}\big)_{t_{j-1},t_{j}}S^{(k)}\big(\rho_{u}^{\frac{n}{N+1}}\big)_{t_{j},t_{j+1}}\big]^{\gamma}\nonumber \\
 & \leq & \frac{c_{6,N}}{\beta}\big[\frac{1}{(n-N-1)!}S^{\left(N+1\right)}\left(\rho_{u}^{\frac{n}{N+1}}\right)_{t_{j-1},t_{j+1}}\big]^{\gamma},\label{eq:after Chen}
\end{eqnarray}
where $c_{6,N}=c_{5,N}(N+1)^{1-\gamma}$. Note that by explicit computation,
there is some constant $c_{N,n}$ independent of $u,v_{1},v_{2}$
such that 
\[
S^{\left(N+1\right)}\left(\rho_{u}^{\frac{n}{N+1}}\right)_{v_{1},v_{2}}=c_{N,n}\left[\left(v_{2}-u\right)^{\frac{n}{N+1}}-\left(v_{1}-u\right)^{\frac{n}{N+1}}\right]^{N+1}.
\]
Since 
\begin{eqnarray*}
 &  & \sum_{j=1}^{r-1}\left(t_{j+1}-u\right)^{\frac{n}{N+1}}-\left(t_{j-1}-u\right)^{\frac{n}{N+1}}\\
 & \leq & 2\left(\left(t-u\right)^{\frac{n}{N+1}}-\left(s-u\right)^{\frac{n}{N+1}}\right),
\end{eqnarray*}
there exists a $j$ such that 
\begin{eqnarray}
 &  & \left(t_{j+1}-u\right)^{\frac{n}{N+1}}-\left(t_{j-1}-u\right)^{\frac{n}{N+1}}\nonumber \\
 & \leq & \frac{2}{r-1}\left(\left(t-u\right)^{\frac{n}{N+1}}-\left(s-u\right)^{\frac{n}{N+1}}\right).\label{eq:drop special point}
\end{eqnarray}
As (\ref{eq:drop special point}) would still hold if we replace $\frac{2}{r-1}$
by $1$, we have 
\begin{eqnarray*}
 &  & \left(t_{j+1}-u\right)^{\frac{n}{N+1}}-\left(t_{j-1}-u\right)^{\frac{n}{N+1}}\\
 & \leq & \left(\frac{2}{r-1}\wedge1\right)\left(\left(t-u\right)^{\frac{n}{N+1}}-\left(s-u\right)^{\frac{n}{N+1}}\right)
\end{eqnarray*}
 Using this particular $j$ in the expression (\ref{eq:after Chen})
as well as the algebraic Lemma \ref{drop point algebra}, we have
\begin{eqnarray*}
 &  & \Vert\sum_{k\geq N+1}X_{u,s}^{n-k}\star\big(X_{s,t}^{\mathcal{P},k}-X_{s,t}^{\mathcal{P}\backslash(t_{j}),k}\big)\Vert_{\mathcal{T},\gamma,\hat{c}_{N}^{-1}\beta}\\
 & \leq & \frac{c_{6,N}}{\beta}\left(\frac{2}{r-1}\wedge1\right)^{\left(N+1\right)\gamma}\frac{1}{\left(n-N-1\right)!^{\gamma}}S^{\left(N+1\right)}\left(\rho_{u}^{\frac{n}{N+1}}\right)_{s,t}^{\gamma}
\end{eqnarray*}
By iteratively removing points and observing that
\[
X_{s,t}^{\left\{ s,t\right\} ,k}=0
\]
for $k\geq N+1$, we have that for all partitions $\mathcal{P}$,
\begin{eqnarray*}
 &  & \Vert\sum_{k\geq N+1}X_{u,s}^{n-k}\star X_{s,t}^{\mathcal{P},k}\Vert_{\mathcal{T},\gamma,\hat{c}_{N}^{-1}\beta}\\
 & \leq & \frac{c_{6,N}}{\beta}\sum_{r=2}^{\infty}\left(\frac{2}{r-1}\wedge1\right)^{\left(N+1\right)\gamma}\frac{1}{\left(n-N-1\right)!^{\gamma}}S^{\left(N+1\right)}\left(\rho_{u}^{\frac{n}{N+1}}\right)_{s,t}^{\gamma}.
\end{eqnarray*}
In particular, 
\begin{eqnarray*}
 &  & \Vert\sum_{k\geq N+1}X_{u,s}^{n-k}\star X_{s,t}^{\mathcal{P},k}\Vert_{\mathcal{T},\gamma,\beta}\\
 & = & \hat{c}_{N}\Vert\sum_{k\geq N+1}X_{u,s}^{n-k}\star X_{s,t}^{\mathcal{P},k}\Vert_{\mathcal{T},\gamma,\hat{c}_{N}^{-1}\beta}\\
 & \leq & \frac{c_{6,N}\hat{c}_{N}}{\beta}\sum_{r=2}^{\infty}\left(\frac{2}{r-1}\wedge1\right)^{\left(N+1\right)\gamma}\frac{1}{\left(n-N-1\right)!^{\gamma}}S^{\left(N+1\right)}\left(\rho_{u}^{\frac{n}{N+1}}\right)_{s,t}^{\gamma}.
\end{eqnarray*}
We have the desired estimate if we let $\left|\mathcal{P}\right|\rightarrow0$
and choose 
\[
\beta\geq c_{6,N}\hat{c}_{N}\sum_{r=2}^{\infty}\left(\frac{2}{r-1}\wedge1\right)^{\left(N+1\right)\gamma}.
\]

\end{proof}
\begin{proof}[Proof of main result Theorem 4]Let $\hat{C}_{N}$ denote
the right hand side of (\ref{eq:beta condition}). For $X\in\mathcal{H}^{*}$,
let $\left\Vert \cdot\right\Vert $ denote the following normalised
Hölder norm of $X$ for degrees up to $N$, 
\[
\left\Vert X\right\Vert =\max_{1\leq\left|\tau\right|\leq N,\tau\;\mbox{trees}}\left\Vert \left\langle X,\tau\right\rangle \right\Vert _{\gamma,\tau}^{\left(\gamma\left|\tau\right|\right)^{-1}}.
\]
where Hölder norm $\Vert\cdot\Vert_{\gamma,\tau}$ of each degree
is define in (\ref{eq:Holder}) in the definition of branched rough
path. Applying Lemma \ref{main lemma}, which we have just proved,
to the branched rough path $Y$ defined for each rooted tree $\tau$
by 
\[
\langle Y_{s,t},\tau\rangle=\frac{1}{\left(N!\left\Vert X\right\Vert \right)^{\gamma\left|\tau\right|}\hat{C}_{N}^{\left|\tau\right|}}\langle X_{s,t},\tau\rangle
\]
 with $\beta=\hat{C}_{N}$, we have by taking $u=s$ that for $|\tau|\geq N+1$,
\begin{eqnarray*}
|\langle X_{u,t},\tau\rangle| & \leq & \frac{1}{\hat{C}_{N}}\left(N!\left\Vert X\right\Vert \right)^{\gamma\left|\tau\right|}\hat{C}_{N}^{\left|\tau\right|}\left[\frac{\left|\tau\right|!}{\tau!\left(\left|\tau\right|-N-1\right)!}S^{\left(N+1\right)}\left(\rho_{u}^{\frac{|\tau|}{N+1}}\right)_{u,t}\right]^{\gamma}\\
 & \leq & \frac{2\exp\left(N+1\right)}{\hat{C}_{N}}\left(N!\left\Vert X\right\Vert \right)^{\gamma\left|\tau\right|}\hat{C}_{N}^{\left|\tau\right|}\frac{\left(t-u\right)^{\gamma\left|\tau\right|}}{\tau!^{\gamma}}.
\end{eqnarray*}
\end{proof}

\section{Appendix: Binomial-type lemmas}
\begin{lem*}
Let $N\in\mathbb{N}\cup\left\{ 0\right\} $ and $n\geq N+1$. For
all $u<s<t$ 
\begin{eqnarray*}
\sum_{j=N+1}^{n}\frac{\left(s-u\right)^{n-j}\left(t-s\right)^{j}}{\left(n-j\right)!j!} & \leq & \frac{1}{\left(n-N-1\right)!}S^{\left(N+1\right)}\left(\frac{N+1}{n}\left(\cdot-u\right)^{\frac{n}{N+1}}\right)_{s,t}.
\end{eqnarray*}
\end{lem*}
\begin{proof}
The following identity can be proved using an induction on $N$ or
Taylor's theorem, 
\begin{eqnarray}
J & := & \sum_{j=N+1}^{n}\frac{\left(s-u\right)^{n-j}\left(t-s\right)^{j}}{\left(n-j\right)!j!}\label{eq:taylor's theorme}\\
 & = & \frac{1}{\left(n-N-1\right)!}\int_{\triangle_{N+1}\left(s,t\right)}\left(s_{1}-u\right)^{n-N-1}\mathrm{d}s_{1}\ldots\mathrm{d}s_{N+1}.
\end{eqnarray}
Note that as we are integrating over the domain $s_{1}<s_{2}\ldots<s_{N+1}$,

\begin{eqnarray*}
J & \leq & \frac{1}{\left(n-N-1\right)!}\int_{\triangle_{N+1}\left(s,t\right)}\Pi_{i=1}^{N+1}\left(s_{i}-u\right)^{\frac{n-N-1}{N+1}}\mathrm{d}s_{1}\ldots\mathrm{d}s_{N+1}\\
 & = & \frac{1}{\left(n-N-1\right)!}S^{\left(N+1\right)}\left(\frac{N+1}{n}\left(\cdot-u\right)^{\frac{n}{N+1}}\right)_{s,t}.
\end{eqnarray*}
\end{proof}
\begin{lem*}
(Binomial lemma for overlapping intervals) Let $N\in\mathbb{N}\cup\left\{ 0\right\} $.
Let $n\geq N+1$ and $m\geq0$, and $u\leq s\leq t$, 
\begin{eqnarray}
 &  & \frac{n!}{\left(n-N-1\right)!}S^{\left(N+1\right)}\left(\frac{N+1}{n}\left(\cdot-u\right)^{\frac{n}{N+1}}\right)_{s,t}\left(t-u\right)^{m}\label{eq:tree multiply-1}\\
 & \leq & \frac{\left(n+m\right)!}{\left(n+m-N-1\right)!}S^{\left(N+1\right)}\left(\frac{N+1}{n+m}\left(\cdot-u\right)^{\frac{n+m}{N+1}}\right)_{s,t}.
\end{eqnarray}
\end{lem*}
\begin{proof}
Using that for any $b\geq a$ and $c\geq d$, $c(b-a)\leq(bc-ad)$,
\begin{eqnarray}
 &  & S^{\left(N+1\right)}\left(\left(\cdot-u\right)^{\frac{n}{N+1}}\right)_{s,t}\left(t-u\right)^{m}\nonumber \\
 & = & \frac{\big[\left(t-u\right)^{\frac{n}{N+1}}-\left(s-u\right)^{\frac{n}{N+1}}\big]^{N+1}\left(t-u\right)^{m}}{\left(N+1\right)!}\nonumber \\
 & \leq & \frac{\big[\left(t-u\right)^{\frac{n+m}{N+1}}-\left(s-u\right)^{\frac{n+m}{N+1}}\big]^{N+1}}{\left(N+1\right)!}\nonumber \\
 & = & S^{\left(N+1\right)}\left(\left(\cdot-u\right)^{\frac{n+m}{N+1}}\right).\label{eq:calcul 1}
\end{eqnarray}
The lemma now follows by noting that as $n+m\geq n$, 
\begin{equation}
\frac{n!}{\left(n-N-1\right)!}\left(\frac{N+1}{n}\right)^{N+1}\leq\frac{\left(n+m\right)!}{\left(n+m-N-1\right)!}\left(\frac{N+1}{n+m}\right)^{N+1}.\label{eq:calcu 2}
\end{equation}
\end{proof}
\begin{lem*}
(The estimate is decreasing) For all $0\leq k\leq m\leq n$ and $u\leq s\leq t$,
\begin{equation}
\frac{1}{\left(n-m\right)!}S^{\left(m\right)}\left(\frac{m}{n}\left(\cdot-u\right)^{\frac{n}{m}}\right)_{s,t}\leq\frac{\exp m}{\left(n-m+k\right)!}S^{\left(m-k\right)}\left(\frac{m-k}{n}\left(\cdot-u\right)^{\frac{n}{m-k}}\right)_{s,t}.\label{eq:decreasing-1}
\end{equation}
\end{lem*}
\begin{proof}
Since for any $p\geq1$ and $a\geq b$, $\left(a-b\right)^{p}\leq a^{p}-b^{p}$,
\begin{eqnarray*}
J^{\prime} & := & \frac{1}{\left(n-m\right)!}S^{\left(m\right)}\left(\frac{m}{n}\left(\cdot-u\right)^{\frac{n}{m}}\right)_{s,t}\\
 & = & \frac{1}{\left(n-m\right)!}\left(\frac{m}{n}\right)^{m}\frac{\left(\left(t-u\right)^{\frac{n}{m}}-\left(s-u\right)^{\frac{n}{m}}\right)^{m}}{m!}\\
 & \leq & \frac{1}{\left(n-m\right)!}\left(\frac{m}{n}\right)^{m}\frac{\left(\left(t-u\right)^{\frac{n}{m-k}}-\left(s-u\right)^{\frac{n}{m-k}}\right)^{m-k}}{m!}
\end{eqnarray*}
As $\frac{m^{m}}{m!}\leq\exp\left(m\right)$, $n^{m}(n-m)!\geq n^{m-k}(n-m+k)!$
and $\frac{(m-k)^{m-k}}{(m-k)!}\geq1$, 
\begin{eqnarray*}
J^{\prime} & \leq & \frac{\exp\left(m\right)}{\left(n-m+k\right)!}\left(\frac{m-k}{n}\right)^{m-k}\frac{\left(\left(t-u\right)^{\frac{n}{m-k}}-\left(s-u\right)^{\frac{n}{m-k}}\right)^{m-k}}{\left(m-k\right)!}\\
 & = & \frac{\exp\left(m\right)}{\left(n-m+k\right)!}S^{\left(m-k\right)}\left(\frac{m-k}{n}\left(\cdot-u\right)^{\frac{n}{m-k}}\right)_{s,t}.
\end{eqnarray*}
\end{proof}

\end{document}